\documentclass[reqno, oneside, 10pt]{article}
\usepackage{amsmath,amsfonts,amssymb,amsthm}

\usepackage{tikz}
\tikzstyle{vertex}=[ultra thin,circle,fill, minimum size=4pt, inner sep=1pt] %describing how vertices look in our pictures
\tikzset{ultra thick/.style={line width=2pt}} %default value is 1.6pt
\usepackage{bbm}
\usepackage{setspace}
\usepackage{geometry}
\usepackage[section]{placeins} %forces floats (e.g. figures) appear in the section in which they are defined

\newtheorem{theorem}{Theorem}
\newtheorem{proposition}[theorem]{Proposition}

\newtheorem{lemma}[theorem]{Lemma}
\newtheorem{claim}[theorem]{Claim}
\newtheorem{remark}{Remark}
\newtheorem{problem}{Problem}

\newcommand\eps{\varepsilon}
\renewcommand\le{\leqslant}
\renewcommand\ge{\geqslant}

\newcommand\E{{\mathbb E}}

\newcommand\Var{\operatorname{Var}} 
 
\renewcommand\Pr{{\mathbb P}}
\newcommand\prob[1]{\Pr\left(#1\right)}

\newcommand{\indic}[1]{\mathbbm{1}_{\{{#1}\}}}

\newcommand\ceil[1]{\lceil #1 \rceil}

\newcommand\Bigceil[1]{\Bigl\lceil#1\Bigr\rceil}

\newcommand\xpar[1]{(#1)}
\newcommand\bigcpar[1]{\bigl\{#1\bigr\}}

\newcommand\bigpar[1]{\bigl(#1\bigr)}
\newcommand\Bigpar[1]{\Bigl(#1\Bigr)}
\newcommand\biggpar[1]{\biggl(#1\biggr)}

\newcommand\Bin{\operatorname{Bin}}

\newcommand{\G}{{\mathbb G}}

\newcommand{\Gnp}{{\mathbb G}_{n,p}} % \newcommand{\Gnp}{\G(n,p)}

\newcommand\cC{\mathcal{C}}
\newcommand\cD{\mathcal{D}}
\newcommand\cE{\mathcal{E}}

\newcommand\cH{{\mathcal H}}

\newcommand\cI{{\mathcal I}}

\newcommand\cP{{\mathcal P}}

\newcommand\cS{{\mathcal S}}
\newcommand\cT{{\mathcal T}}

\newcommand{\fH}{{\mathfrak H}}

\newcommand\whp{whp} % LW: looks nicer

\newcommand{\e}{{e}} %we no longer use $e$ for number of edges, so can restore the standard notation for the constant
\newcommand{\xx}{{\mathbf x}} 
\newcommand{\y}{{\mathbf y}} 
\newcommand{\yy}{{\mathbf y}} 
\newcommand{\ii}{{\mathbf x}}

\newcommand{\aut}{\operatorname{aut}}

\newcommand{\zerost}{$0$-statement}
\newcommand{\onest}{$1$-statement}
\newcommand{\oset}[1]{[n]_{#1}} %number of ordered sets of cardinality #1
\newcommand{\osets}{\oset{v_G}}
\newcommand{\dex}{z} %number of disjoint extensions in the strictly balanced proof
\newcommand{\extcount}[1]{\hat N_{#1}} %number of disjoint extensions in the strictly balanced proof
\newcommand{\vr}[2]{{v_{#2}-v_{#1}}}
\newcommand{\er}[2]{{e_{#2}-e_{#1}}}
\newcommand{\vGH}{\vr{G}{H}} %notation for $v(G,H) = v_H - v_G$
\newcommand{\eGH}{\er{G}{H}} %notation for $e(G,H) = e_H - e_G$

\newcommand{\refT}[1]{Theorem~\ref{#1}}

\newcommand{\refL}[1]{Lemma~\ref{#1}}
\newcommand{\refR}[1]{Remark~\ref{#1}}
\newcommand{\refS}[1]{Section~\ref{#1}}

\newcommand{\refP}[1]{Proposition~\ref{#1}}
\newcommand{\refPr}[1]{Problem~\ref{#1}}

\newcommand{\refF}[1]{Figure~\ref{#1}}
\newcommand{\refApp}[1]{Appendix~\ref{#1}}

\newcommand{\refCl}[1]{Claim~\ref{#1}}

\newcommand\noproof{\qed}

\newenvironment{romenumerate}[1][-0.25em]{% optional argument unused
\vspace{-0.25em}
\vspace{#1}\begin{enumerate}% gives (i), (ii) etc.
\itemsep0pt \parskip0pt \parsep0pt% 
 }{\end{enumerate}\vspace{-0.75em}}
\newenvironment{romenumerate2}[1][-0.25em]{% optional argument unused
\vspace{#1}\begin{enumerate}% gives (i), (ii) etc.
\itemsep0pt \parskip0pt \parsep0pt% 
 }{\end{enumerate}\vspace{-0.05em}}

\newcounter{case}

\newcounter{thmenumerate}

\let\OLDthebibliography\thebibliography
\renewcommand\thebibliography[1]{
  \OLDthebibliography{#1}
  \setlength{\parskip}{0pt}
  \setlength{\itemsep}{0pt plus 0.3ex}
}

\begin{document}
\title{Counting extensions revisited} 
\author{Matas {\v{S}}ileikis\thanks{Institute of Computer Science of the Czech Academy of Sciences, Pod Vod\'arenskou v\v{e}\v{z}\'{i} 2, 182~07~Prague, Czech Republic. 
E-mail: {\tt matas.sileikis@gmail.com}. With institutional support RVO:67985807. Research supported by the Czech Science Foundation, grant number GA19-08740S.} \ 
 and 
Lutz Warnke\thanks{School of Mathematics, Georgia Institute of Technology, Atlanta GA~30332, USA.
E-mail: {\tt warnke@math.gatech.edu}. Research partially supported by NSF Grant DMS-1703516 and a Sloan Research Fellowship.}}
\date{7 November 2019; revised 28 July 2021} 
\maketitle

\begin{abstract}
We consider rooted subgraphs in random graphs, i.e., extension counts such as 
(i)~the number of triangles containing a given~vertex or (ii)~the number of paths of length three connecting two given~vertices. \linebreak[3] 
In~1989, Spencer gave sufficient conditions for the event that, with high probability, these extension counts are asymptotically equal for all choices of the root~vertices.  
For the important strictly balanced case, Spencer also raised the fundamental question as to whether these conditions are necessary. 
We answer this question by a careful second moment argument, and discuss some intriguing problems that remain~open. 
\end{abstract}

%\tableofcontents

\section{Introduction}
Subgraph counts and their many natural generalizations are central topics in random graph theory: 
since the~1960's they are a constant source of beautiful problems and~conjectures, 
which have repeatedly inspired the development of important new probabilistic~techniques and~insights (see~\cite{BB,AS,JLR,FK}).

In this paper we consider \mbox{rooted subgraph counts} in the binomial random graph~$\Gnp$, 
i.e., so-called \linebreak[4] \mbox{extension counts}~\cite{SS1988,S90b,LS1991,Vu2001}  
such as 
(i)~the number of triangles containing a given~vertex  
or (ii)~the number of paths of length three connecting two given~vertices. % (we defer the precise definitions to \refS{sec:intro:main}). 
In combinatorics and related areas, the need for studying such extension~counts %\mbox{(rooted~subgraph) extension~counts} 
arises frequently in probabilistic proofs and applications, 
including zero-one~laws in random graphs~\cite{SS1988,LS1991,S2001}, 
games on random graphs~\cite{LP2010,N2017}, 
random graph processes~\cite{bohman2010early,BFL2015, bohman2013dynamic, pontiveros2013triangle, BW2018}, 
sparse random analogues of classical extremal and Ramsey results~\cite{NS2017,SS2018,BK2019}, 
and many more, such as~\cite{S90a,R92,Vu2001,ST2002,YR2007,JR11,spohel2013general,M2015,TBD}. 
Consequently the investigation of extension~counts is not only a natural problem in probabilistic combinatorics, %random graph theory, 
but also an important issue from the applications point of~view.

% \enlargethispage{\baselineskip} %\enlargethispage{\baselineskip}

After initial groundwork of Shelah and Spencer~\cite{SS1988} as well as Spencer~\cite{S90a} on (rooted~subgraph)  extension counts, 
%\linebreak[3]
in~1989 Spencer~\cite{S90b} proved sufficient conditions for the event that, with high probability\footnote{As usual, we say that an event holds~\emph{whp} (with~high~probability) if it holds with probability tending to~$1$ as~$n\to \infty$.}, 
these extension counts are asymptotically equal in~$\Gnp$ for all choices of the root~vertices.  
For the important strictly balanced case, he also raised the fundamental question whether these sufficient conditions (see~\eqref{eq_Spencer_SB}~below) are qualitatively necessary. 
In this paper we answer Spencer's 30-year old question by a careful second moment argument (see~Theorem~\ref{thm_strictly_balanced}~below),  
rectifying a surprising gap in the random graph literature.  
We also discuss some further partial results and intriguing open~problems (see~Sections~\ref{ss_partial}--\ref{sec:intro:general}~below).

\subsection{Main result}\label{sec:intro:main} 
To fix notation, by a \emph{rooted graph}~${(G,H)}$ 
we mean a graph~${H=(V(H),E(H))}$ and an induced subgraph~${G \subseteq H}$ 
with labeled `root' vertices~${V(G)=\{1, \ldots, v_G\}}$. 
Given a tuple~${\xx =(x_1, \ldots, x_{v_G})}$ of distinct vertices from some `host' graph, 
a~\emph{$(G,H)$-extension} of $\xx$ is a copy of the graph~${H_G:= (V(H), E(H)\setminus E(G))}$ in which each vertex~${j \in V(G)}$ is mapped onto~$x_j$. 
Note that if~$\xx$ spans a copy of~$G$ in the host graph (i.e.,~if the function~${j \mapsto x_j}$ maps edges of $G$ to edges in the host graph),  
then every~$(G,H)$-extension of~$\xx$ corresponds to a copy of~$H$.
Since the edges between root vertices do not affect the definition of a~$(G,H)$-extension, the reader may without loss of generality assume that~$V(G)$ is an independent set~of~$H$ in the results below, cf.~\cite{JLR,JR11} (allowing for~$G$ that are not independent will be convenient in some proofs, though). 
For brevity, we write~$\osets$ for the set of all \emph{roots}, i.e., tuples~$\xx = (x_1, \dots, x_{v_G})$ of distinct vertices from~$[n] := \left\{ 1, \dots, n \right\}$.  
Let~$X_{\xx} = X_{G,H}(\xx)$ denote the number of~$(G,H)$-extensions of~$\xx$ in the binomial random graph~$\Gnp$. 
Note that the expected~value 
\begin{equation}\label{def:muGH}
  \mu=\mu_{G,H} := \E X_{\xx} \asymp n^{v_H-v_G}p^{e_H-e_G}
\end{equation}
does not depend\footnote{Here~$a_n \asymp b_n$ is a convenient shorthand for~$a_n = \Theta(b_n)$, following standard asymptotic notation as in~\cite[p.~9]{JLR}.} on the particular choice of~$\xx$.   
To avoid trivialities, we henceforth assume that~$H$ has more edges than~$G$, i.e., that~$e_H>e_G$. 
Extending the standard density notation for unrooted subgraphs, we define
\begin{equation}\label{def:mGH}
m(G,H) := \max_{G \subsetneq J \subseteq H} d(G,J) \quad \text{ with } \quad d(G,J):=\frac{e_J-e_G}{v_J-v_G},
\end{equation}
and say that~$(G,H)$ is \emph{strictly~balanced} if~$d(G,J) < d(G,H)$ for all~$G \subsetneq J \subsetneq H$. 
We also call~$(G,H)$~\emph{grounded} if at least one root vertex~$j \in V(G)$ is connected to a non-root vertex~$w \in V(H) \setminus V(G)$.

Spencer derived in~1989 sufficient conditions for the event that, with high probability, 
all extension counts satisfy~$X_{\xx} \sim \mu$, i.e., are asymptotically equal. % for all roots~$\xx \in \osets$. 
In the important case when~$(G,H)$ is strictly balanced, 
\mbox{\cite[Theorem~2]{S90b}} states that for every fixed~$\eps \in (0,1]$ there is a constant~$K(\eps) > 0$ such that 
\begin{equation}\label{eq_Spencer_SB}
	\lim_{n \to \infty} \Pr\Bigpar{\max_{\xx \in \osets}|X_\xx - \mu| < \eps\mu} = 1 \quad \text{ if } \mu \ge K(\eps) \log n. %, 
\end{equation}
%where the logarithmic term appears in order to guarantee concentration of~$X_\xx$ at all~$\Theta(n^{v_G})$ roots~$\xx$. 
%
Spencer remarked that %in~\eqref{eq_Spencer_SB}
 his constant satisfies~$K(\eps) \to \infty$ as~$\eps \to 0$, 
and speculated that this is probably also necessary, see~\cite[Remark~on~p.249]{S90b}. 
In other words, he raised the question whether his sufficient condition is qualitatively 
best possible. 

Our main result answers this fundamental question: %  for strictly balanced extension counts% for \mbox{(rooted~subgraph) extension~counts}: 
\eqref{eq:main:strbal}~shows that the `correct' dependence is~$K(\eps)=\Theta(\eps^{-2})$ in the grounded case, even when~$\eps=\eps(n) \to 0$ at some polynomial~rate.
For completeness, \eqref{eq:main:strbal:non}~also shows that the logarithm in the sufficient condition~\eqref{eq_Spencer_SB} is~unnecessary in the less interesting ungrounded~case (where extension counts are essentially unrooted subgraph counts, cf.~example~(b) in~\refF{fig_primal}). 
\begin{theorem}[Main result: strictly balanced case]\label{thm_strictly_balanced}% Strictly balanced case
Let~$(G,H)$ be a rooted graph that is strictly balanced. 
There are constants~$c, C, \alpha > 0$ such that, for all~$p=p(n) \in [0,1]$ and~$\eps=\eps(n) \in [n^{-\alpha},1]$, the following~holds: 
\begin{romenumerate2}
\item% [(i)]
If the rooted graph~$(G,H)$ is grounded, then 
\begin{align}
\label{eq:main:strbal}
\vspace{-0.125em}
	\lim_{n \to \infty} \Pr\Bigpar{\max_{\xx \in \osets}|X_\xx - \mu| < \eps\mu}  &= 
  \begin{cases}
	  0 & \text{if $\eps^2\mu \le c \log n$,} \\
    1 & \text{if $\eps^2\mu \ge C \log n$.}
  \end{cases}\hspace{2.5em}\vspace{-0.125em}
\intertext{\item%[(ii)]
If the rooted graph~$(G,H)$ is not grounded, then}
\label{eq:main:strbal:non}
	\vspace{-0.25em}
	\lim_{n \to \infty} \Pr\Bigpar{\max_{\xx \in \osets}|X_\xx - \mu| < \eps\mu}  &= 
  \begin{cases}
	  0 & \text{if $\eps^2\mu \to 0$,} \\
    1 & \text{if $\eps^2\mu \to \infty$.}
  \end{cases}\hspace{2.5em}\vspace{-0.125em}%
\end{align}%
\end{romenumerate2}\vspace{-0.125em}%
\end{theorem}
\noindent
In concrete words, \eqref{eq:main:strbal}--\eqref{eq:main:strbal:non} of~\refT{thm_strictly_balanced} give 
thresholds for the concentration of extension counts in terms of~$\eps^2\mu$, 
similar to the thresholds in terms of the edge probability~$p$ that are well-known %to exist 
for many properties of~$\Gnp$. 
The role of the expression~$\eps^2 \mu$ in~\eqref{eq:main:strbal}--\eqref{eq:main:strbal:non} %of~\refT{thm_strictly_balanced}
can be made plausible by pretending that $X_\xx$ behaves like a binomial random variable with expectation~$\mu$ (the actual behaviour is of course more involved), 
in which case Chernoff-type tail bounds of the form~$\Pr(|X_\xx - \mu| \ge \eps \mu) \le \e^{-\Omega(\eps^2\mu)}$ hold.
Indeed, considering the union bound over the~$\Theta(n^{v_G})$ roots~$\xx$, it then seems plausible that the $1$-statement follows when~$\eps^2\mu$ is at least a large enough multiple of~$\log n$. % see~\eqref{eq:main:strbal}. 
An intuitive reason why the~$\log n$ factor is absent in the ungrounded threshold~\eqref{eq:main:strbal:non} is that here the~$X_{\xx}$ are strongly correlated and in fact almost equal (e.g., in example~(b) from~\refF{fig_primal} each~$X_{\xx}$ is well-approximated by the total number of triangles), so there should be no need to use a union~bound.

The main contribution of \refT{thm_strictly_balanced} 
is the $0$-statement in the grounded threshold~\eqref{eq:main:strbal}, 
which was missing in previous work: % previous results such as~\eqref{eq_Spencer_SB} above:  
our proof uses a careful second moment argument 
(combining correlation inequalities and counting arguments with Janson's inequality) 
in order to establish that, with high probability, 
there exists a root~$\xx$ with~${X_\xx \ge (1+\eps)\mu}$, i.e., with too~many $(G,H)$-extensions. 
This is closely related to the task of obtaining good lower bounds on ${\Pr(X_\xx \ge (1+\eps)\mu)}$, which are not so well understood as upper bounds; % (based on concentration inequalities); 
see~\cite{JR2002,JW,Ch19,SW18}. % and the references therein. 
To sidestep this conceptual obstacle, in \refS{s_strictly_balanced} we therefore work with (easier to estimate)   
auxiliary events that enforce~${X_\xx \ge (1+\eps)\mu}$ via `disjoint' extensions, 
and we believe that our approach might also be useful for establishing `lower bounds' in other~problems.

\pagebreak[3]

\subsection{Partial results: beyond the strictly balanced case}\label{ss_partial}
We also establish some threshold results for extension counts of rooted graphs~$(G,H)$ that are not necessarily strictly balanced. 
Here things are more complicated, since we now need to take into account all subgraphs~${J \subseteq H}$ containing the root~$G$, 
in particular those that satisfy~$d(G,J) = m(G,H)$; cf.~\cite{S90a,S90b,R92,JLR}. 
We call such subgraphs~$J$~\emph{primal}, 
and for brevity also say that~$J$~is~\emph{grounded} if~$(G,J)$~is~grounded. 
The partial results Theorems~\ref{thm_unique}--\ref{thm_nogrounded} below cover all strictly balanced~$(G,H)$, 
and they in particular imply that \refT{thm_strictly_balanced} also holds with~$\eps^2\Phi$ instead of~$\eps^2\mu$ (possibly after modifying the constants~$c,C,\alpha$), where 
\begin{equation}\label{eq_PhiGH}
	\Phi = \Phi_{G,H} := \min_{G \subseteq J \subseteq H :  e_J > e_G} \mu_{G,J}.
\end{equation}
There is no contradiction here: the extra assumption~$\eps \ge n^{-\alpha}$ ensures that the conclusions of the {$0$-}~and {$1$-statements} of \refT{thm_strictly_balanced} coincide regardless of whether we use~$\eps^2\Phi$ or~$\eps^2\mu$ (cf.~\refS{sec:ext:non}).   
It thus comes as no surprise that in our main result \refT{thm_strictly_balanced} the technical assumption~$\eps \ge n^{-\alpha}$ is indeed\footnote{For examples~(a) and~(b) from \refF{fig_primal} with~$\eps \asymp n^{-1/2}$ and~$\eps \asymp n^{-1}$, when~$p \asymp n^{-1/4}$ it is routine to check that~$\Phi \to \infty$, $\eps^2\Phi \to 0$ and~$\eps^2 \mu \gg \log n$ in both cases. Hence the~$0$-statement holds by~\eqref{eq:general} of \refT{thm_general}, showing that \eqref{eq:main:strbal}--\eqref{eq:main:strbal:non} of~\refT{thm_strictly_balanced}~fail.} 
necessary. 
%
%\footnote{When~$p \asymp n^{-1/4}$ and~$\eps \asymp n^{-1/2}$, then for graph~(a) in \refF{fig_primal} the~$0$-statement holds by~$\eps^2\Phi \asymp \eps^2 \min\{n^2p^3,np\} \to 0$ and~\eqref{eq:general}, 
%which together with~$\eps^2 \mu \asymp \eps^2n^2p^3 \gg \log n$ shows that~\eqref{eq:main:strbal}~fails.  
%When~$p \asymp n^{-1/4}$ and~$\eps \asymp n^{-1}$, then for graph~(b) in \refF{fig_primal} the~$0$-statement holds by~$\eps^2\Phi \asymp \eps^2 \min\{n^3p^3,n^2p\} \to 0$ and~\eqref{eq:general}, 
%which together with~$\eps^2 \mu \asymp \eps^2n^3p^{3} \to \infty$ shows that~\eqref{eq:main:strbal:non}~fails.} 

The following result covers the case where $(G,H)$ has only one primal subgraph which also happens to be grounded, such as in examples~(a) and~(c) from~\refF{fig_primal};  
this case includes the rooted graphs in~\refT{thm_strictly_balanced}~(i) since in that case~$H$ is a unique primal subgraph. 
% which implies that the threshold~\eqref{eq:main:strbal} also holds with~$\eps^2\Phi$ instead of~$\eps^2\mu$ (possibly after modifying the constants~$c,C,\alpha$). 
%
\begin{theorem}[Unique and grounded primal case]\label{thm_unique}%
Let~$(G,H)$ be a rooted graph with a unique primal subgraph~$J$. 
If~$(G,J)$ is grounded, then there are constants $c, C, \alpha > 0$ such that,  
for all~$p=p(n) \in [0,1]$ and $\eps=\eps(n) \in [n^{-\alpha},1]$, 
\begin{equation}\label{eq:thm:unique}
	\lim_{n \to \infty} \Pr\Bigpar{\max_{\ii \in \osets} |X_{\ii} - \mu| < \eps\mu} = 
  \begin{cases}
	  0 &\text{if } \eps^2 \Phi \le c \log n, \\
		1 &\text{if } \eps^2 \Phi \ge C \log n.
	\end{cases}\vspace{-0.125em}%
\end{equation}%
\end{theorem}
\noindent
The heuristic idea is that the main contribution to deviations of~$X_\xx=X_{G,H}(\xx)$ comes from those of~$X_{G,J}(\xx)$, 
and, since~$(G,J)$ is strictly balanced and grounded, the problem thus intuitively reduces to \refT{thm_strictly_balanced}~(i).

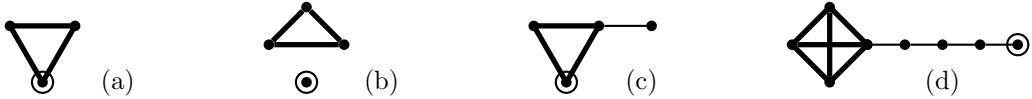
\begin{figure}
  \begin{center}
	  \hspace*{\fill}
    \begin{tikzpicture}[thick,scale=1]
	    %\clip (-1,-0.7) rectangle (1.5,1); %an invisible frame; note that things beyond it become invisible
	\foreach \x in {1, 2, ..., 3}
	{
		\draw[ultra thick] (-120*\x+30:0.5) node[vertex] (n\x) {} -- (-120*\x + 150:0.5);
	}
	\draw  (n1) circle (4pt);
	\node  at (1,-0.5) () {(a)};
    \end{tikzpicture}
	  \hspace*{\fill}
    \begin{tikzpicture}[thick,scale=1]
	    %\clip (-1,-0.7) rectangle (1.5,1); %an invisible frame; note that things beyond it become invisible
	\foreach \x in {1, 2, ..., 4}
	{
		\node[vertex] at (-90*\x:0.5) (n\x) {};
	}
	\foreach \x/\y in {2/3,3/4,4/2}
	{
		\draw[ultra thick] (n\x) -- (n\y);
	}
	\draw  (n1) circle (4pt);
	\node  at (1,-0.5) () {(b)};
	%\node  at (0,-1.5) () {(b)};
    \end{tikzpicture}
	  \hspace*{\fill}
    \begin{tikzpicture}[thick]
	    %\clip (-1,-0.7) rectangle (1.5,1); %an invisible frame
	\foreach \x in {1, 2, ..., 3}
	{
		\draw[ultra thick] (-120*\x+30:0.5) node[vertex] (n\x) {} -- (-120*\x + 150:0.5);
	}
	\draw  (n1) circle (4pt);
	\draw (n3) -- +(0.7,0) node[vertex] (){};
	\node  at (1,-0.5) () {(c)};
	%\node at (0,-1.5) () {(c)};
    \end{tikzpicture}
	  \hspace*{\fill}
    \begin{tikzpicture}[thick]
	    %\clip (-1,-0.7) rectangle (3,1); %an invisible frame
	\begin{scope}
		\foreach \x in {1, 2, ..., 4}
		{
			\draw[ultra thick] (-90*\x:0.5) node[vertex] (n\x) {} -- (-90*\x + 90:0.5);
		}
	\end{scope}
	\foreach \x/\y in {1/3,2/4}
	{
		\draw[ultra thick] (n\x) -- (n\y);
	}
	\foreach \x/\y in {4/5, 5/6, 6/7, 7/8, 8/9}
	{
		\draw (n\x) -- (0.5*\x - 1.5, 0) node[vertex] (n\y) {};
	}
	\draw  (n9) circle (4pt);
	\node  at (1.5,-0.5) () {(d)};
	%\node at (1,-1.5) () {(d)};
    \end{tikzpicture}\vspace{-1.125em}
	  \hspace*{\fill}
  \end{center}
  \caption{Examples of rooted graphs, with the root vertex circled and primal subgraphs marked in bold: 
(a)~strictly balanced and grounded, 
(b)~strictly balanced and not~grounded, 
% non-balanced  
(c)~with a unique primal that is~grounded, 
and 
% non-balanced 
(d)~with a unique primal that is not~grounded.
Our main result \refT{thm_strictly_balanced} applies to~(a),(b), 
\refT{thm_unique} applies to~(a),(c), 
\refT{thm_nogrounded} applies to~(b),(d), 
and \refT{thm_general} applies to all of~them.}
		\label{fig_primal}
\end{figure}

The following result covers the case where no primal subgraph of~$(G,H)$ is grounded, such as in examples~(b) and~(d) from~\refF{fig_primal};  
this case includes the rooted graphs in~\refT{thm_strictly_balanced}~(ii). 
%, which means that the threshold~\eqref{eq:main:strbal:non} also holds with~$\eps^2\Phi$ instead of~$\eps^2\mu$ (possibly after modifying the constant~$\alpha$). 
%
\begin{theorem}[No grounded primals case]\label{thm_nogrounded}%
Let~$(G,H)$ be a rooted graph with no grounded primal subgraphs. 
There is a constant~$\alpha > 0$ such that, 
for all~$p=p(n) \in [0,1]$ and~$\eps=\eps(n) \in [n^{-\alpha},1]$, 
\begin{equation}\label{eq:main:nogrounded}
	\lim_{n \to \infty} \Pr\Bigpar{\max_{\ii \in \osets} |X_{\ii} - \mu| < \eps\mu} = 
  \begin{cases}
	  0 &\text{if } \eps^2 \Phi \to 0, \\
		1 &\text{if } \eps^2 \Phi \to \infty.
	\end{cases}\vspace{-0.125em}%
\end{equation}%
\end{theorem}
\noindent
Similar to~\refT{thm_strictly_balanced}~(ii), 
the intuition is that all~$X_\xx$ are approximately equal once we know the number of unrooted copies of a certain subgraph of~$H$ 
(e.g., in example~(d) from \refF{fig_primal} this special subgraph~is~$K_4$).

Theorems~\ref{thm_unique}--\ref{thm_nogrounded} give thresholds for the concentration of extension counts in terms of~$\eps^2\Phi$. 
For general~${(G,H)}$ we do not have such a threshold, but the following %approximate 
result intuitively states that the transition from the \mbox{$0$-statement} to the \mbox{$1$-statement} 
always happens at some point as~$\eps^2 \Phi$ changes from~$o(1)$ to~$n^{\Omega(1)}$.  
%even without the prevailing technical assumption~$\eps \ge n^{-\alpha}$. 
%
%The mild assumptions~$1-p = \Omega(1)$ and~$\Phi \to \infty$ below are both fairly~natural 
%(for example, $\Phi \to 0$ implies the~$0$-statement; see~\refT{thm_generaltail} in~\refS{s_prelim}). 
%
\begin{theorem}[General case: approximate conditions]\label{thm_general}% 
Let~$(G,H)$ be a rooted graph. 
For all~$p=p(n) \in [0,1]$ and~$\eps  = \eps(n) \in (0,1]$ with~$1-p = \Omega(1)$ and~$\Phi \to \infty$, 
\begin{equation}\label{eq:general}
	\lim_{n \to \infty} \Pr\Bigpar{\max_{\xx \in \osets}|X_\xx - \mu| < \eps\mu} = 
  \begin{cases}
	  0 &\text{if } \eps^2 \Phi \to 0, \\
		1 &\text{if } \eps^2 \Phi  = n^{\Omega(1)}.
	\end{cases}\vspace{0.75em}%\vspace{-0.125em}% % hack to increase spacing between theorem and footnote
\end{equation}%
\end{theorem}
\noindent
The $1$-statement in~\eqref{eq:general} %of~\refT{thm_general} 
implies~\cite[Corollary~4]{S90b}, which in turn strengthens a result that 
played a key role in the study of zero-one laws~\cite{SS1988} due to Shelah and Spencer
(since the `safe' assumptions from~\cite{S90b,SS1988} imply~$\Phi = n^{\Omega(1)}$ 
via~\refR{rem:Phibig}~\ref{eq:Phibig:iv} from \refS{s_prelim}).

\pagebreak[3]

\subsection{Discussion: open problems and cautionary examples}\label{sec:intro:general}
For rooted subgraph extension counts, the main open problem is to fully determine the thresholds for concentration, %{$0$-}~and {$1$-statement} conditions, 
i.e., to close the gap in~\eqref{eq:general} of \refT{thm_general} 
(and to weaken the conditions of Theorems~\ref{thm_strictly_balanced}--\ref{thm_nogrounded}).   
\begin{problem}\label{prb:open}%
Determine the `correct' conditions for the $0$-~and $1$-statements of any rooted graph~$(G,H)$. 
\end{problem}
%
%\noindent
%
Our %general 
understanding of~\refPr{prb:open} is still far from satisfactory. % (even on a heuristic level).  
Indeed, even for fixed~$\eps \in (0,1]$ the correct {$1$-statement} condition remains open, 
which we now illustrate for the rooted graph~(e) from \refF{counterexample}. 
In this case, any {$(G,H)$-extension} can be viewed as a combination of a {$(G,K_4)$-extension} and a {$(K_4,H)$-extension}. 
The proof of Spencer's general $1$-statement \mbox{result~\cite[Theorem~3]{S90b}} combines this 
decomposition with his strictly balanced result~\eqref{eq_Spencer_SB} for~$(G,K_4)$ and~$(K_4,H)$, 
leading to a sufficient condition of form~$\min\{\mu_{G,K_4}, \mu_{K_4,H}\} \ge K'(\eps) \log n$ (cf.~\cite[Section~2]{S90b}). 
The following result shows that this sufficient condition can be weakened in some range, demonstrating that Spencer's general $1$-statement condition is not always optimal. 
%
%\footnote{Applying our main result~\eqref{eq:main:strbal} instead of~\eqref{eq_Spencer_SB} 
%leads to a sufficient condition of form~$\eps^2\min\{\mu_{G,K_4}, \mu_{K_4,H}\} \ge C' \log n$, 
%and \refP{prop:counterexample:2} shows that also this refined condition is not optimal, even when we allow for~$\eps=\eps(n) \to 0$.}
%
\begin{proposition}\label{prop:counterexample:2}%
Let~$(G,H)$ be the rooted graph~(e) depicted in \refF{counterexample}. 
Set~$\omega := np^2$. 
For all~$p=p(n) \in [0,1]$ and $\eps=\eps(n) \in (0,1]$ 
such that~$\omega \ll \log n$ and~$\eps^2 \omega^3 \gg \log n$, 
we have~$\eps^2\mu_{G,K_4} \gg \log n \gg \eps^2 \mu_{K_4,H}$ 
but~$\Pr(\max_{\xx \in \osets}|X_\xx - \mu| < \eps\mu) \to 1$ as~$n \to \infty$. 
\end{proposition}
\begin{figure}
  \begin{center}
    \hspace*{\fill}
    \begin{tikzpicture}[thick,scale=0.5]
	%\clip (-1.5,-1.5) rectangle (4,2); %an invisible frame
	\draw 
	(-1,0) node[vertex] (root) { } -- (1, 1) node[vertex] (top) { } -- (1, -1) node[vertex] (bot) { } -- (root)
	(top) -- (0.2, 0) node[vertex] (mid) {} -- (bot)
	(mid) -- (root)
	(top) -- (2, 0) node[vertex] (add1) {} -- (bot);
	\draw  (root) circle (8pt);
	\node at (3.5,-1) () {(e)};
    \end{tikzpicture}
    \hspace*{\fill}
    \begin{tikzpicture}[thick,scale=0.5]
	%\clip (-1.5,-1.5) rectangle (4,2); %an invisible frame
	\draw 
	(-1,0) node[vertex] (root) { } -- (1, 1) node[vertex] (top) { } -- (1, -1) node[vertex] (bot) { } -- (root)
	(top) -- (0.2, 0) node[vertex] (mid) {} -- (bot)
	(mid) -- (root)
	(top) -- (2, 0.5) node[vertex] (add1) {} -- (bot)
	(top) -- (2, -0.5) node[vertex] (add2) {} -- (bot);
	\draw  (root) circle (8pt);
	\node at (3.5,-1) () {(f)};
    \end{tikzpicture}\vspace{-1.125em}
    \hspace*{\fill}
  \end{center}
  \caption{The rooted graphs used in~Propositions~\ref{prop:counterexample:2}--\ref{prop:counterexample:1}, with the root vertex circled: 
for~(e) Spencer's general $1$-statement is not~optimal, 
and for~(f) the natural condition~$\eps^2\Phi \gg \log n$ does~not imply the~$1$-statement.}
		\label{counterexample}
\end{figure}
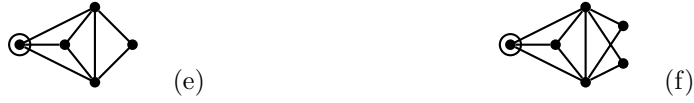
\noindent
It is not hard to see that in the setting of \refP{prop:counterexample:2} we have $\eps^2\Phi \asymp \eps^2 \mu_{G, K_4} \gg \log n$, which together with Theorems~\ref{thm_unique}--\ref{thm_nogrounded} suggests that maybe~$\eps^2 \Phi \gg \log n$ is always a sufficient condition\footnote{Further support comes from the fact that~$X_\xx$ is asymptotically normal, see~\refCl{cl:mom}~\ref{cl:mom:asymp} in \refApp{apx:general} and the variance estimate~\eqref{eq:Variance} from \refS{s_prelim}, which makes it plausible that~$\Pr(|X_\xx - \mu| \ge \eps\mu) \le \e^{-\Omega((\eps \mu)^2/\Var X_\xx)} \le \e^{-\Omega(\eps^2 \Phi)} \ll n^{-v_G}$ holds, which in turn would then establish the $1$-statement by taking the union bound over all~$\Theta(n^{v_G})$ roots~$\xx$.} 
for the $1$-statement (which would sharpen~\refT{thm_general}). 
However, the following result shows that this speculation is false for the rooted graph~(f) depicted in~\refF{counterexample}, 
indicating that Problem~\ref{prb:open} is more tricky than one might~think.
\begin{proposition}\label{prop:counterexample:1}
Let~$(G,H)$ be the rooted graph~(f) depicted in \refF{counterexample}. 
Set~$\omega := np^2$. 
For all~$p=p(n) \in [0,1]$ and~$\eps=\eps(n) \in (0,1]$ 
such that~$\omega \ll (\log n)^{0.39}$ and~$\eps^2 \omega^3 \gg \log n$, 
we have~$\eps^2\Phi \asymp \eps^2\mu_{G,K_4} \gg \log n$ 
but~$\Pr(\max_{\xx \in \osets}|X_\xx - \mu| < \eps\mu) \to 0$ as~$n \to \infty$. 
\end{proposition}
%
%\refP{prop:counterexample:1} also implies that the strictly balancedness assumption in Theorem~\ref{thm_strictly_balanced} is necessary
%(for the rooted graph~(f) from \refF{counterexample} the~$0$-statement and~$\eps^2 \mu \ge \eps^2\Phi \gg \log n$ both hold, 
%which shows that~\eqref{eq:main:strbal}~fails).

Overall, we hope that the above intriguing examples and open problems   
will stimulate more research into rooted subgraph counts. %(see also \refS{sec:conclusion}). 
When~$(G,H)$ is strictly balanced and grounded, 
then we conjecture that~\eqref{eq:thm:unique} holds for suitable~$c,C>0$ 
under the natural assumptions~$\mu \to \infty$ and~$1-p=\Omega(1)$, 
i.e., without assuming~$\eps \ge n^{-\alpha}$. 
We leave it as an open problem to formulate a conjecture for the general solution to~\refPr{prb:open}, 
%(the special case when all primals are grounded already appears to be interesting), 
which in many cases is closely related to determining the regime where $\Pr(|X_\xx - \mu| \ge \eps\mu)$ changes from~$n^{-o(1)}$ to~$n^{-\omega(1)}$, say. 
In the concluding remarks we also discuss a potential connection to extreme value theory (see \refS{sec:conclusion}).

%\enlargethispage{\baselineskip} %\enlargethispage{\baselineskip}

\subsection{Organization of the paper}
In \refS{s_prelim} we introduce some auxiliary results, which also imply~\refT{thm_general}. 
In \refS{s_strictly_balanced} we prove our main result \refT{thm_strictly_balanced}~(i) for strictly balanced~$(G,H)$ that are grounded. 
In Sections~\ref{s_nogrounded} and~\ref{s_unique} we prove Theorems~\ref{thm_unique} and~\ref{thm_nogrounded}, 
i.e., cover the case where no grounded primal of~$(G,H)$ exists, and the case where the primal of~$(G,H)$ is unique and grounded, respectively.  
In \refS{sec:ext:non}, we prove \refT{thm_strictly_balanced}~(ii) for strictly balanced~$(G,H)$ that are not grounded. 
In \refS{sec:counterexample} we prove the cautionary examples from Propositions~\ref{prop:counterexample:2}--\ref{prop:counterexample:1}. 
Finally, \refS{sec:conclusion} contains some concluding remarks and~problems.  
%Finally, the appendix contains some calculations that are omitted from the main~text. 

\pagebreak[3]

\section{Preliminaries}\label{s_prelim}
In this section we collect some useful basic observations, 
and a partial result which implies Theorem~\ref{thm_general}. 
First, by adapting the textbook argument~\cite[Lemma~3.5]{JLR} for (unrooted) subgraph counts, 
for any rooted graph~$(G,H)$ it is standard to see that the variance of~$X_{G,H}(\xx)$ satisfies 
\begin{equation}\label{eq:Variance}
\sigma^2 = \sigma_{G,H}^2 := \Var X_{G,H}(\xx) 
	\asymp (1 - p)\mu_{G,H}^2 /\Phi_{G,H}
\end{equation}
for any edge probability~$p=p(n) \in (0,1]$, 
where~$\mu=\mu_{G,H}$ and~$\Phi=\Phi_{G,H}$ are as defined in~\eqref{def:muGH} and~\eqref{eq_PhiGH};~cf.~\cite{Matas2012phd}. 
Next, inspired by similar statements for subgraph counts~\cite[Lemma~3.6]{JLR}, 
using the relation~$\mu_{G,J} \asymp \xpar{n^{1/d(G,J)}p}^{e_J-e_G}$ for all~$G \subseteq J \subseteq H$ with~$e_J > e_G$,  
it is straightforward to establish the following useful~properties. 
Recall that~$m(G,H)$ and $\Phi = \Phi_{G,H}$ are defined in~\eqref{def:mGH} and~\eqref{eq_PhiGH}, respectively. 
\begin{remark}\label{rem:Phibig}%
For any rooted graph~$(G,H)$, the following hold for all~$p=p(n)\in [0,1]$:%
\begin{romenumerate}
\item\label{eq:Phibig:i}% [(i)]
$\Phi \to \infty$ is equivalent to~$p \gg n^{-1/m(G,H)}$.
\item\label{eq:Phibig:ii}% [(ii)] 
$\Phi = \Omega(1)$ is equivalent to~$p = \Omega(n^{-1/m(G,H)})$. 
\item\label{eq:Phibig:iii}%[(iii)] 
If~$\Phi \asymp 1$, then~$\mu_{G,J} \asymp 1$ for any~$G \subseteq J \subseteq H$ that is primal for~$(G,H)$.
\item\label{eq:Phibig:iv}% [(iv)] 
If~$p = \Omega(n^{-1/m(G,H) + \eta})$ for some constant~$\eta \ge 0$, then~$\Phi = \Omega(n^{\eta})$. 
\end{romenumerate}
\end{remark}
\noindent
Finally, the approximate result Theorem~\ref{thm_general} immediately follows 
from the following slightly more general theorem, 
whose technical statement will be convenient in several later proofs. 
In particular, in some ranges of the parameters, we will be able to deduce the desired $1$-~or $0$-statements directly from~\eqref{eq:thm_generaltail:1}--\eqref{eq:thm_generaltail:0} below. 
\begin{theorem}\label{thm_generaltail}%
For any rooted graph~$(G,H)$, the following hold for all~$p=p(n)\in [0,1]$:%  
\begin{romenumerate}
\item\label{thm_tail1}%
	If~$\Phi = \Omega(1)$ and $(t/\mu)^2\Phi\ge n^{\Omega(1)}$, then 
	\begin{equation}\label{eq:thm_generaltail:1}
\lim_{n \to \infty} \Pr\Bigpar{\max_{\xx \in \osets}|X_\xx - \mu| < t} = 1. \hspace{2.5em}
	\end{equation}%
\item\label{thm_tail0}%
	If $\eps = \eps(n) \in (0,1]$ and either (a)~$\Phi(1-p) \to \infty$ and $\eps^2 \Phi/(1-p) \to 0$, or~(b)~$\Phi \to 0$, then 
	\begin{equation}\label{eq:thm_generaltail:0}
\lim_{n \to \infty} \Pr\Bigpar{\max_{\xx \in \osets}|X_\xx - \mu| \ge \eps \mu} = 1 . \hspace{2.5em}
\vspace{-0.125em}%
	\end{equation}%
\end{romenumerate}
\end{theorem}
\begin{remark}\label{rem:thm_generaltail}%
In~\ref{thm_tail1}, the conclusion~\eqref{eq:thm_generaltail:1} holds with probability~$1 - o(n^{-\tau})$ for any constant~$\tau > 0$.
\end{remark}
%
%\begin{proof}[Proof of Theorem \ref{thm_general}]
	%Follows immediately from Theorem~\ref{thm_generaltail}.
%\end{proof}
%
\noindent
We defer the simple proof of Theorem~\ref{thm_generaltail} to Appendix~\ref{apx:general}, and only mention the main ideas here. 
Claim~\ref{thm_tail0} exploits that~$X_\xx$ is asymptotically normal when $\Phi(1-p) \to \infty$. % (under very mild assumptions). 
Claim~\ref{thm_tail1} is based on Markov's inequality and a central moment estimate $\E (X_\xx - \mu)^{2m} \le C_m \sigma^{2m} \le D_m (\mu^2/\Phi)^m$ that is a by-product of the usual asymptotic normality proof via the method of moments (see Claim~\ref{cl:mom} in Appendix~\ref{apx:general}).  
This approach for obtaining tail estimates `without much effort' does not seem to be as widely known in probabilistic combinatorics, 
and we believe that it will be useful in other applications
(e.g., it yields a simple direct proof of~\cite[Corollary~4]{S90b}). %
%which in turn strengthens a result that played a key role in the study of zero-one laws~\cite{SS1988}). 

\section{Strictly balanced and grounded case (Theorem~\ref{thm_strictly_balanced})}\label{s_strictly_balanced}
In this section we prove the threshold~\eqref{eq:main:strbal} of Theorem~\ref{thm_strictly_balanced}~(i) for strictly balanced rooted graphs~$(G,H)$ that are grounded (see \refS{sec:ext:non} for the less interesting ungrounded case). 

The~$0$-statement in~\eqref{eq:main:strbal} is the main difficulty, and here the plan is to 
use a second moment argument to show the existence of a root~$\xx \in \osets$ with too many $(G,H)$-extensions, i.e., with~$X_{\xx} \ge (1+\eps)\mu$. 
Unfortunately, even an asymptotic estimate of the relevant first moment is challenging, 
since the upper tail proba\-bi\-li\-ty~$\Pr(X_{\xx} \ge (1+\eps)\mu)$ is hard to estimate up to a~$1+o(1)$ factor 
(this is an instance of the `infamous' upper tail problem~\cite{JR2002,SW18}). 
To sidestep this technical difficulty, we instead show the existence of a root~${\xx \in \osets}$ 
which attains~$X_{\xx}=\ceil{(1+\eps)\mu}$ due to exactly~$\ceil{(1+\eps)\mu}$ extensions that are vertex-disjoint outside of~$\xx$. 
The crux is that these auxiliary events are more tractable: we can estimate the relevant first and second moments up to the required~$1+o(1)$ factors  
via a careful mix of Harris' Lemma~\cite{Harris}, Janson's inequality~\cite{J90,BS,RiordanWarnke2015}, and counting arguments. 
It turns out that here the extra assumption~$\eps \ge n^{-\alpha}$ is helpful: 
it will allow us to focus on fairly small edge probabilities~$p=p(n)$ that are close to~$n^{-1/d(G,H)}$,  
which intuitively makes it easier to show that various events are approximately independent (as tacitly required by the second moment method);  
see~\refS{sec:0statement} for the~details.

The~$1$-statement in~\eqref{eq:main:strbal} is simpler (and nowadays fairly routine). 
For edge probabilities~$p=p(n)$ that are close to~$n^{-1/d(G,H)}$, we use a standard union bound argument, estimating the lower tail~$\Pr(X_{\xx} \le (1-\eps)\mu)$ via Janson's inequality~\cite{J90,JLR,RiordanWarnke2015} and the upper tail~$\Pr(X_{\xx} \ge (1+\eps)\mu)$ via an inequality of Warnke~\cite{WUT}.
For edge probabilities~$p=p(n)$ much larger than~$n^{-1/d(G,H)}$, 
it turns out that we can simply use the partial result \refT{thm_generaltail}~\ref{thm_tail1} due to the extra assumption~$\eps \ge n^{-\alpha}$;  
see~\refS{sec:1statement} for the~details.

\subsection{Technical preliminaries}\label{s_strictly_balanced:prelim}
Our upcoming arguments exploit two standard properties of strictly balanced rooted graphs: 
(i)~for fairly small edge probabilities~$p=p(n)$, the expectation~$\mu=\mu_{G,H}$ is significantly smaller than any other expectation~$\mu_{G,J}$ with~$G \subsetneq J \subsetneq H$ (note that~$\mu_{G,H}/\mu_{G,J} \asymp n^{v_H - v_J}p^{e_H - e_J} \ll 1$ via~\eqref{eq:lem:density:subs} below), 
and (ii) after removing the root vertices from~$H$, the remaining graph~${H - V(G)}$ is connected. 
Both mimic well-known properties from the unrooted case, 
so we defer the routine %density based 
proof of \refL{lem:StrBal} to~\refS{s_strictly_balanced:prelim:deferred}. 
\begin{lemma}\label{lem:StrBal}%
For any strictly balanced rooted graph~$(G,H)$, the following hold:% 
\begin{romenumerate}
\item\label{eq:StrBal:density}% [(i)]
There is a constant $\beta = \beta(G,H) > 0$ such that, for all~$p=p(n) \in [0,1]$ with~$p= O(n^{-1/d(G,H) + \beta})$, 
\begin{equation}
\label{eq:lem:density:subs}
\max_{G \subsetneq J \subsetneq H} n^{v_H - v_J}p^{e_H - e_J} \ll n^{-\beta}. \hspace{2.5em}
\vspace{-0.125em}%
\end{equation}%
\item\label{eq:StrBal:connected}% [(ii)] 
The graph~${H - V(G)}$, obtained from~$H$ by deleting the vertices of~$G$, is connected. 
\end{romenumerate}
\end{lemma}

\subsection{The $0$-statement}\label{sec:0statement} 
Our second-moment-based proof of the \zerost{} in~\eqref{eq:main:strbal} of Theorem~\ref{thm_strictly_balanced} hinges on the following key lemma.
Given a root~$\xx \in \osets$, let~$\cE_{\xx}$ denote the event that, in~$\Gnp$, the root~$\xx$ has 
exactly~$\dex:= \ceil{(1 + \eps)\mu}$ many $(G,H)$-extensions, and all of them are pairwise vertex-disjoint (i.e., sharing no vertices outside~$\xx$). 
We also say that two roots~$\xx_1, \xx_2 \in \osets$ are \emph{disjoint} if they share no elements as (unordered) sets. 
\begin{lemma}\label{lem:main}%
Let $(G,H)$ be a rooted graph that is strictly balanced and grounded. 
There are constants~$c, \gamma > 0$ such that, 
for all~$\eps=\eps(n) \in (0,1]$ and~$p=p(n) \in [0,1]$ with~$p \le n^{-1/d(G,H) + \gamma}$, $\mu \ge 1/2$ and~$\eps^2\mu \le c \log n$, 
the following holds: 
for all roots~$\xx \in \osets$ we have 
\begin{equation}\label{eq:pr:lb}
\Pr(\cE_{\xx}) \gg n^{-1/2},
\end{equation}
and for all disjoint roots~$\xx_1,\xx_2 \in \osets$ we have 
\begin{equation}\label{eq:pr:ub}
\Pr(\cE_{\xx_1}, \: \cE_{\xx_2}) \le (1 + o(1)) \Pr(\cE_{\xx_1})\Pr(\cE_{\xx_2}).
\end{equation}
\end{lemma}
\begin{proof}[Proof of the \zerost{} in~\eqref{eq:main:strbal} of Theorem~\ref{thm_strictly_balanced}] % (assuming Lemma~\ref{lem:main}) 
Let~$c, \gamma>0$ be the constants given by Lemma~\ref{lem:main}. 
Fix arbitrary $0 < \alpha < \gamma/2$.
First, when~$p > n^{-1/d(G,H) + \gamma}$, then~$\eps \ge n^{-\alpha}$ and \refR{rem:Phibig}~\ref{eq:Phibig:iv} 
imply~$\eps^2\mu \ge {n^{-2\alpha} \cdot \Phi_{G,H}} = \Omega(n^{\gamma - 2\alpha}) \gg \log n$, so the condition of the \zerost{} cannot be satisfied and hence there is nothing to prove. 
Next, when~$\mu < 1/2$, then~$(1+\eps) \mu \le 2 \mu < 1$ and~$\eps \le 1$ imply that the interval~$\left((1-\eps)\mu, (1 + \eps)\mu\right)$ contains no integers, and so the \zerost{} again holds trivially. 

Thus we can henceforth assume~$\mu \ge 1/2$ and~$p \le n^{-1/d(G,H) + \gamma}$, as required by Lemma~\ref{lem:main}. 
For convenience, we set~$s := \lfloor n / v_G \rfloor \asymp n$, and choose disjoint roots~$\xx_1, \dots, \xx_s \in \osets$. 
Writing~$Y := |\left\{ i \in [s] : \cE_{\xx_i} \text { holds} \right\}|$, to prove the \zerost{} of Theorem~\ref{thm_strictly_balanced} we shall now show that~$Y > 0$~\whp{}, i.e., that~$\Pr(Y>0) \to 1$ as~$n \to \infty$. 
Using~\eqref{eq:pr:lb} we obtain~$\E Y = \sum_{1 \le i \le s}\Pr(\cE_{\xx_i}) \gg s \cdot n^{-1/2} \asymp n^{1/2} \to \infty$. 
Together with~\eqref{eq:pr:ub} it follows~that 
\begin{equation*}\label{eq:mu2:ub}
\begin{split}
\E Y^2 & \le \sum_{1 \le i ,j \le s: \; i \neq j} \Pr(\cE_{\xx_i}, \: \cE_{\xx_j}) + \sum_{1 \le i \le s} \Pr(\cE_{\xx_i})  \; \le \; (1 + o(1)) \cdot (\E Y)^2 + \E Y \; \sim \; (\E Y)^2. 
\end{split}
\end{equation*}
Now Chebyshev's inequality readily yields $\Pr(Y=0) \le \Var Y/(\E Y)^2 \to 0$ as~$n \to \infty$, completing the~proof. 
\end{proof}

The remainder of Section~\ref{sec:0statement} is dedicated to the proof of Lemma~\ref{lem:main}. 
For concreteness, for~$\beta>0$ as given by \refL{lem:StrBal}~\ref{eq:StrBal:density}, we choose the constants~$\gamma, c \in (0,1/2)$ such that  
\begin{equation}\label{eq:gammadef}
\gamma e_H \; < \; \min\bigcpar{\beta/v_H, \: \beta/2, \: 1/2, \: 1-2c}. 
\end{equation}
Recalling~$\mu \asymp n^\vGH p^\eGH$ and~$\eps \le 1$, using the assumptions~$\mu \ge 1/2$ and~$p \le n^{-1/d(G,H) + \gamma}$, we infer 
\begin{equation}\label{eq:mumupper}
1/2 \: \le \: \mu \: \le \: \dex = \ceil{(1 + \eps)\mu} \: \le \: O(n^{\gamma e_H}) \ll \min\bigcpar{n^{1/2},n^{\beta/2}} , 
\end{equation}
and
\begin{equation}\label{eq:p_half}
p \le \Bigpar{n^{-(v_H-v_G) + \gamma(\eGH)}}^{\frac{1}{\eGH}} \le \Bigpar{n^{-1 + 1/2}}^{\frac{1}{\eGH}} \ll 1/2,
\end{equation}
with room to spare. 
With foresight, given~$\xx \in \osets$, we denote by~$N = N_{G,H}(\xx)$ the number of $(G,H)$-extensions of~$\xx$ in~$K_n$. 
Note that $N \asymp n^{\vGH}$ does not depend on the particular choice of~$\xx$.

\subsubsection{The first moment: inequality~\eqref{eq:pr:lb}}\label{sec:first}
We start with~\eqref{eq:pr:lb}, i.e., a lower bound for~$\Pr(\cE_{\xx})$. 
Recall that every $\xx \in \osets$ has~$N$ extensions in~$K_n$. 
The plan is to show that~$\Pr(\cE_{\xx})$ is comparable with $\Pr(\Bin(N,p^\eGH) = \dex)$. 
More precisely, we will show that 
\begin{equation}\label{eq:pr:lb:bin}
\Pr(\cE_{\xx}) \; \ge \; (1+o(1)) \cdot \binom{N}{\dex} p^{(\eGH)\dex} (1-p^\eGH)^{N - \dex} .
\end{equation}
In view of $\dex \approx (1+\eps)\mu = (1 + \eps) Np^\eGH$, using Stirling's formula 
it then will be routine to deduce that the lower bound in~\eqref{eq:pr:lb:bin} is~$\Theta(z^{-1/2}) \cdot \e^{-\Theta(\eps^2\mu)}$, 
which together with~\eqref{eq:gammadef}--\eqref{eq:mumupper} and the assumption~$\eps^2\mu \le c\log n$ 
will eventually imply the desired inequality~\eqref{eq:pr:lb}; see \eqref{eq:MoivreLaplace}--\eqref{eq:MoivreLaplace2}~below.

Turning to the technical details, given $\xx \in \osets$, 
let~$\fH(\xx)$ denote the set of all (unordered) collections of~$\dex = \ceil{(1 + \eps)\mu}$ vertex-disjoint $(G,H)$-extensions of~$\xx$ in~$K_n$. 
Given~$\cC \in \fH(\xx)$, let~$\cC^c$ denote the remaining~${N - \dex}$ extensions of~$\xx$ in~$K_n$. 
Given a collection~$\cS$ of extensions of~$\xx$, we write~$\cI_{\cS}$ for the event that all extensions in~$\cS$ are present in~$\Gnp$, 
and~$\cD_{\cS}$ for the event that all extensions in~$\cS$ are not present in~$\Gnp$.
Note that  
\begin{equation}\label{eq:er}
  \Pr(\cE_{\xx}) = \sum_{\cC \in \fH(\xx)} \Pr(\cI_{\cC} , \: \cD_{\cC^c}) = \sum_{\cC \in \fH(\xx)} \Pr(\cI_{\cC}) \Pr(\cD_{\cC^c} \mid \cI_{\cC}) \ge |\fH(\xx)| \min_{\cC \in \fH(\xx)} \Pr(\cI_{\cC})\Pr(\cD_{\cC^c} \mid \cI_{\cC}),
\end{equation}
where the minimum is of course only formal: by symmetry the probabilities are the same for every~$\cC \in \fH(\xx)$.
To estimate~$|\fH(\xx)|$, note that given~$i \le z$ vertex-disjoint extensions, the number of choices for another vertex-disjoint extension is $N - O(z n^{v_H - v_G - 1})$.
Since~$\fH(\xx)$ consists of unordered collections of extensions, 
using~$N \asymp n^{\vGH}$ and~$\dex \ll n^{1/2}$ (see~\eqref{eq:mumupper}) together with~$1 - x = e^{-x(1 + o(1))}$ as~$x \to 0$ 
it follows~that
\begin{equation}\label{eq:Hr}
  |\fH(\xx)| = \frac{\left(N- O(\dex n^{\vGH-1})\right)^\dex}{\dex!} 
	= \frac{N^\dex}{\dex!} \cdot \Bigpar{1-O\bigpar{z/n}}^z
	%= \frac{N^\dex}{\dex!} \cdot \e^{O(\dex^2/n)} %\exp\Bigpar{O(\dex^2/n)} 
\sim \frac{N^\dex}{\dex!} \sim \binom{N}{\dex} . 
\end{equation}
Since the extensions in $\cC \in \fH(\xx)$ are disjoint, we have 
\begin{equation}\label{eq:I}
\Pr(\cI_{\cC})= p^{(\eGH)\dex}.
\end{equation}
For the remaining lower bound on~$\Pr(\cD_{\cC^c} | \cI_{\cC})$, 
the idea is to apply Harris' Lemma~\cite{Harris} 
and then use \refL{lem:StrBal}~\ref{eq:StrBal:density} to show that the effect of `overlapping' pairs of extensions is negligible. 
\begin{claim}\label{cl:D:lower}%
Let~$\xx \in \osets$. Then, for all~$\cC \in \fH(\xx)$, we have 
\begin{equation}\label{eq:D:lower}
\Pr(\cD_{\cC^c} | \cI_{\cC}) \; \ge \; (1+o(1)) \cdot (1-p^\eGH)^{N-\dex} .
\end{equation}
\end{claim}
\begin{proof}%
We fix~$\cC \in \fH(\xx)$, and define the auxiliary graph~$F := \bigpar{[n], \; \bigcup_{H_1 \in \cC} E(H_1)}$. 
Note that after conditioning on the event~$\cI_\cC$, in~$\Gnp$ each possible edge from~$E(K_n) \setminus E(F)$ is still included independently with probability~$p$. 
Therefore Harris' Lemma (see, e.g.,~{\cite[Theorem~6.3.2]{AS}}) implies that 
\begin{equation} \label{eq:Harris0} 
  \Pr(\cD_{\cC^c} | \cI_{\cC}) \ge \prod_{H_2 \in \cC^c} \bigpar{1-p^{\eGH - e(H_2 \cap F)}}.
\end{equation}
Note that there are at most ${N - z}$ extensions~$H_2 \in \cC^c$ with~$e(H_2 \cap F) = 0$, 
  each contributing a factor of~${1 - p^{e_H - e_G}}$ to the right-hand side of~\eqref{eq:Harris0}. 
Every other extension~$H_2 \in \cC^c$ contains at least one edge not in~$F$
(since by \refL{lem:StrBal}~\ref{eq:StrBal:connected}, after deleting the root vertices~$\xx$, all graphs in~$\{H_1 -\xx: H_1 \in \cC\}$ are vertex-disjoint and connected), 
so that~$p^{e_H - e_G - e(H_2 \cap F)} \le p \le 1/2$ by~\eqref{eq:p_half}. 
Since~$1-x \ge \e^{-2x}$ for~$x \le 1/2$, from~\eqref{eq:Harris0} it follows~that
\begin{equation} \label{eq:Harris} 
\Pr(\cD_{\cC^c} | \cI_{\cC}) \ge  (1-p^{\eGH})^{N - \dex} \cdot \exp \Big( -2 \hspace{-0.25em} \sum_{\substack{H_2 \in \cC^c: \\ e(H_2 \cap F) \ge 1}} \hspace{-0.5em} p^{\eGH - e(H_2 \cap F)}\Big) .
\end{equation}
To estimate the sum in \eqref{eq:Harris}, note that if~$H_2 \in \cC^c$ shares an edge with~$F$, 
then~$E(H_2 \cap F)$ corresponds to a $(G,J)$-extension of~$\xx$ for some~$G \subsetneq J \subsetneq H$. 
The number of such extensions is at most~$(v_H\dex)^{v_J-v_G} = O(\dex^{v_H})$, with room to spare. 
Given a $(G,J)$-extension, it can be further extended to some~$H_2 \in \cC^c$ in at most~$n^{v_H-v_J}$ ways. 
Using~$e_H-e_G-(e_J-e_G)=e_H-e_J$ together with~\eqref{eq:mumupper} and \eqref{eq:lem:density:subs}, 
it follows that 
\begin{equation}\label{eq:Harris:overlap}
\sum_{\substack{H_2 \in \cC^c: \\ e(H_2 \cap F) \ge 1}} \hspace{-0.5em} p^{\eGH - e(H_2 \cap F)} \le
\sum_{G \subsetneq J \subsetneq H} \hspace{-0.375em} O\Bigpar{\dex^{v_H} n^{v_H-v_J} \cdot p^{e_H-e_J}} 
\ll n^{\gamma e_H v_H - \beta} = o(1), 
\end{equation}
which together with~\eqref{eq:Harris} establishes inequality~\eqref{eq:D:lower}. 
\end{proof}

Combining estimates~\eqref{eq:er}--\eqref{eq:D:lower}, we readily obtain inequality~\eqref{eq:pr:lb:bin}. 
To establish~\eqref{eq:pr:lb}, it remains to estimate the right-hand side of~\eqref{eq:pr:lb:bin} 
via the following well-known form of Stirling's formula (see, e.g.,~\cite[equation~(1.4)]{BB}):
\begin{equation}\label{eq:Stirling}
  n! = \sqrt{2\pi n}\left( \frac{n}{e} \right)^n e^{\alpha_n} \quad \text{ with } \quad \alpha_n = O(n^{-1}).
\end{equation}
With foresight, let~$t:=\dex-\mu=\eps \mu+O(1)$, and define~$\varphi(x):=(1+x)\log(1+x)-x$ for~$x > -1$.
Recalling~\eqref{eq:mumupper} we have~$1 \le z \ll n^{1/2} \ll N$. 
Using Stirling's formula~\eqref{eq:Stirling} together with~$\mu=Np^{\eGH}$ and~$z=\mu+t$, 
then a simple (but slightly tedious) calculation along the lines of the Appendix of~\cite{W14} gives
\begin{equation}\label{eq:MoivreLaplace}
\begin{split}
  \binom{N}{\dex} p^{(\eGH) \dex} (1-p^{\eGH})^{N-\dex} & \ge \frac{\exp\Bigl(-O\bigl(N^{-1}+\dex^{-1}+(N-\dex)^{-1}\bigr)\Bigr)}{\sqrt{2\pi \dex (1-\dex/N)}} \cdot \left(\frac{\mu}{\dex}\right)^{\dex} \left( \frac{N-\mu}{N-\dex} \right)^{N-\dex} \\
  & \ge \Omega(\dex^{-1/2}) \cdot \exp\Bigl(-\mu \varphi\bigl(t/\mu\bigr) - (N-\mu)\varphi\bigl(-t/(N-\mu)\bigr)\Bigr) .
\end{split}
\end{equation}
Note that~$\log(1+x) \le x$ implies~$\varphi(x) \le x^2$. 
Using~\eqref{eq:mumupper} we readily infer~$t^2/\mu %= (\eps^2\mu^2 + O(1))/\mu 
= \eps^2\mu + O(1)$. 
Furthermore, \eqref{eq:gammadef} implies~$p^\eGH \le n^{-1/2}$, so that~$N \ge n^{1/2}\mu \gg \mu$. 
Using the estimates~\eqref{eq:mumupper} and~$\eps^2\mu \le c \log n$ together with~$\gamma e_H/2 + c < 1/2$ (see~\eqref{eq:gammadef}), 
it now follows that~\eqref{eq:MoivreLaplace} is at least
\begin{equation}\label{eq:MoivreLaplace2}
\Omega\bigpar{z^{-1/2}} \cdot \exp \Bigpar{-\bigpar{1+O\bigpar{n^{-1/2}}} \eps^2\mu} 
	\; \ge \; 
	\Omega(1) \cdot \exp \Bigpar{-\bigpar{\gamma e_H/2 + c}\log n} \gg n^{-1/2},
\end{equation}
which together with~\eqref{eq:pr:lb:bin} completes the proof of inequality~\eqref{eq:pr:lb} from Lemma~\ref{lem:main}. %\noproof 

\subsubsection{The second moment: inequality~\eqref{eq:pr:ub}}
Now we turn to~\eqref{eq:pr:ub}, i.e., an upper bound for $\Pr\left( \cE_{\xx_1}, \cE_{\xx_2} \right)$ when~$\xx_1$, $\xx_2$ are disjoint. 
Recalling~\eqref{eq:er}, note~that 
\begin{equation}\label{eq:EE00}
\Pr(\cE_{\xx_1} , \: \cE_{\xx_2}) = \sum_{\cC_1 \in \fH(\xx_1)}\sum_{\cC_2 \in \fH(\xx_2, \cC_1)} \Pr(\cI_{\cC_1 \cup \cC_2} , \: \cD_{\cC_1^c \cup \cC_2^c}) ,
\end{equation}
where we (with foresight) define
\begin{equation}\label{eq:HRC}
	\fH(\xx_2, \cC_1) := \bigcpar{ \cC_2 \in \fH(\xx_2) : \ \Pr(\cI_{\cC_1 \cup \cC_2} , \: \cD_{\cC_1^c \cup \cC_2^c})>0 }.
\end{equation}
Guided by the heuristics that the various events are approximately independent, 
the plan is to show~that 
\begin{equation}\label{eq:EE01}
\Pr(\cI_{\cC_1 \cup \cC_2} , \: \cD_{\cC_1^c \cup \cC_2^c})
\; \le \; (1+o(1))  \Pr(\cI_{\cC_1} , \: \cD_{\cC_1^c}) \cdot \Pr(\cI_{\cC_2}, \: \cD_{\cC_2^c}) ,
\end{equation}
though the actual details will be slightly more involved. 
Ignoring these complications for now, note that~\eqref{eq:EE01} would together with~\eqref{eq:EE00}, \eqref{eq:er} and~$\fH(\xx_2, \cC_1) \subseteq \fH(\xx_2)$ indeed imply the desired inequality~\eqref{eq:pr:ub}.

Turning to the technical details, 
%since $\cI_{\cC_1 \cup \cC_2}$ is an increasing event and $\cD_{\cC_1^c \cup \cC_2^c}$ is a decreasing event, using~\eqref{eq:EE00} and Harris' Lemma~\cite{Harris} we obtain 
by applying Harris' Lemma (noting that~$\cI_{\cC_1 \cup \cC_2}$ is an increasing event and that~$\cD_{\cC_1^c \cup \cC_2^c}$ is a decreasing event; see~the definitions above~\mbox{\cite[Theorem~6.3.2]{AS}}) 
to the right-hand side of~\eqref{eq:EE00} we obtain~that  
\begin{equation}\label{eq:EE}
\begin{split}
\Pr(\cE_{\xx_1}, \: \cE_{\xx_2})  \le \sum_{\cC_1 \in \fH(\xx_1)}\sum_{\cC_2 \in \fH(\xx_2, \cC_1)} \Pr(\cI_{\cC_1 \cup \cC_2}) \Pr(\cD_{\cC_1^c \cup \cC_2^c}) .
\end{split}
\end{equation}
Recalling that every~$\xx \in \osets$ has~$N$ extensions in~$K_n$, 
Harris' Lemma also gives the lower bound~$\Pr(\cD_{\cC_1^c \cup \cC_2^c}) \ge (1 - p^\eGH)^{2(N-\dex)}$. 
We will now prove an asymptotically matching upper bound that 
does \emph{not} depend on the choice of $\cC_1$ and~$\cC_2$ (similarly as in Claim~\ref{cl:D:lower}). 
Here the idea is to apply a form of Janson's inequality~\cite{BS,JLR,AS}, % (in a form proved by Boppana and Spencer~\cite{BS}), 
and then again use \refL{lem:StrBal}~\ref{eq:StrBal:density} to argue that `overlaps' have negligible contribution. 
\begin{claim}\label{cl:D2:upper}%
Let~$\xx_1,\xx_2 \in \osets$ be disjoint. 
Then, for all~$\cC_1 \in \fH(\xx_1)$ and~$\cC_2 \in \fH(\xx_2)$, we~have 
\begin{equation}\label{eq:D2:upper}
\Pr(\cD_{\cC_1^c \cup \cC_2^c}) \; \le \; (1 + o(1)) \cdot (1-p^\eGH)^{2(N - \dex)} .
\end{equation}
\end{claim}
\begin{proof}%
Let~$\cS$ be the family of edge-sets, each of size~${e_H - e_G}$, corresponding to extensions in~${\cC_1^c \cup \cC_2^c}$ 
(each extension of~$\xx_1$ or~$\xx_2$ is uniquely determined by its edge-set, since~$H$ has no isolated vertices outside of~$V(G)$ by \refL{lem:StrBal}~\ref{eq:StrBal:connected}). 
Note that if an extension in~$\cC_1^c$ is also an extension in~$\cC_2^c$, then it must contain some vertex from~$\xx_2$ (because~$(G,H)$ is grounded).
Since~$\xx_1, \xx_2$ are disjoint, the number of such duplicate extensions is~$O(n^{v_H - v_G - 1})$, which implies that~$|\cS| \ge 2(N - \dex) - O(n^{v_H - v_G - 1})$. 
Setting~$X := \sum_{E \in \cS} \indic{E \subseteq \Gnp}$, note that the event~$\cD_{\cC_1^c \cup \cC_2^c}$ is precisely the event that~$X = 0$.
Since~$p \le 1/2$ (see~\eqref{eq:p_half}) implies~$1/(1-p^{\eGH}) \le 2$ and~$(1-p^\eGH)^{-1} \le e^{2p^\eGH}$, 
by invoking the Boppana--Spencer~\cite{BS} variant of Janson's inequality (see, e.g.,~\cite[Remark~2.20]{JLR} or~\cite[Theorem~8.1.1]{AS})  
it then follows~that
\begin{equation}
\label{eq:Janson}
\Pr(\cD_{\cC_1^c \cup \cC_2^c}) = \Pr\left( X = 0 \right)
\le (1-p^\eGH)^{|\cS|} \cdot \e^{\Delta/(1 - p^{e_H - e_G})} 
\le (1-p^\eGH)^{2(N - \dex)} \cdot \e^{O(n^{v_H - v_G - 1}p^{\eGH}+\Delta)}, 
\end{equation} 
where 
\begin{equation}\label{eq:Janson:Delta}
\Delta := \sum_{\substack{(E_1, E_2) \in \cS\times\cS:\\ 1 \le |E_1 \cap E_2| < \eGH}} \hspace{-0.75em} p^{|E_1 \cup E_2|}.
\end{equation}
Using~$\mu=Np^\eGH \asymp n^{\vGH}p^\eGH$ together with~\eqref{eq:mumupper}, it follows that 
\begin{equation}\label{eq:Janson:approx}
	n^{v_H - v_G - 1}p^\eGH \asymp \mu \cdot n^{-1} \ll n^{1/2 - 1} = o(1).
\end{equation}
Turning to the~$\Delta$-term, note that~$|\cS|p^\eGH \le 2(N - \dex)p^\eGH \le 2 \mu$. 
By proceeding analogously to the estimates in~\eqref{eq:Harris}--\eqref{eq:Harris:overlap}, using \eqref{eq:mumupper} and \eqref{eq:lem:density:subs} it routinely follows~that 
\begin{equation}\label{eq:Janson:Delta:bound}
\Delta 
\le \sum_{E_1 \in \cS} p^{\eGH}\hspace{-0.75em}\sum_{\substack{E_2 \in \cS:\\ 1 \le |E_1 \cap E_2| < \eGH}}\hspace{-1.0em} p^{\eGH-|E_1 \cap E_2|} 
\le O\Bigpar{\mu \cdot \sum_{G \subsetneq J \subsetneq H} n^{v_H - v_J}p^{e_H - e_J}} 
= o(1),
\end{equation}
which together with~\eqref{eq:Janson}--\eqref{eq:Janson:approx} establishes inequality~\eqref{eq:D2:upper}. 
\end{proof}

To sum up, by inserting the estimates~\eqref{eq:I} and~\eqref{eq:D2:upper} into~\eqref{eq:EE}, we readily arrive at
\begin{equation}\label{eq:EE1}
\Pr(\cE_{\xx_1},\cE_{\xx_2}) \; \le \; (1+o(1)) \cdot p^{(\eGH)\dex}(1-p^\eGH)^{2(N - \dex)}  \sum_{\cC_1 \in \fH(\xx_1)}\sum_{\cC_2 \in \fH(\xx_2, \cC_1)} \Pr(\cI_{\cC_2} \mid \cI_{\cC_1}) .
\end{equation}
Anticipating that the main contribution comes from pairs~$\cC_1, \cC_2$ of `disjoint' collections, we~partition 
\begin{equation}\label{eq:part}
\fH(\xx_1) := \fH_0(\xx_1,\xx_2) \cup \fH_{\ge 1}(\xx_1,\xx_2) ,
\end{equation}
where $\fH_0(\xx_1,\xx_2)$ contains the collections~$\cC_1 \in \fH(\xx_1)$ for which the auxiliary~graph 
\begin{equation}\label{eq:F}
  F = F(\cC_1) := \Bigl([n], \; \textstyle \bigcup_{H' \in \cC_1}E(H')\Bigr)
\end{equation}
contains no extensions of~$\xx_2$, and~$\fH_{\ge 1}(\xx_1,\xx_2)$ contains the remaining ones.  
Since~$\xx_1, \xx_2$ are disjoint and~$(G,H)$ is grounded, 
every $\cC_1 \in \fH_{\ge 1}(\xx_1,\xx_2)$ must contain at least one extension overlapping with~$\xx_2$ (in at least one vertex). 
From \eqref{eq:Hr}, $N \asymp n^{\vGH}$ and~$\dex \ll n$ (see~\eqref{eq:mumupper}) 
it follows~that, for some constant~$A = A(G,H) > 0$, 
\begin{equation}\label{eq:negligbleC1}
\left|\fH_{\ge 1}(\xx_1,\xx_2)\right| \le A n^{\vGH-1} \cdot \binom{N}{\dex-1} 
\asymp n^{\vGH-1}\cdot \frac{\dex}{N} \cdot |\fH(\xx_1)|
\ll |\fH(\xx_1)| .
\end{equation}
Exploiting the groundedness assumption, we next show that pairs~$\cC_1, \cC_2$ 
can only overlap in at most~$v_G=O(1)$ extensions (see~Claim~\ref{cl:finiteOverlaps}), 
and that overlapping pairs effectively have negligible contribution (see~Claim~\ref{cl:conditionalsum}). 
\begin{claim}\label{cl:finiteOverlaps}%
Let~$\xx_1,\xx_2 \in \osets$ be disjoint. 
Then, for all~$\cC_1 \in \fH(\xx_1)$, the graph~$F = F(\cC_1)$ defined in~\eqref{eq:F} contains 
at most~$v_G$ vertex-disjoint extensions of~$\xx_2$. 
\end{claim}
\begin{proof}%
The graph~${F - \xx_1}$, obtained by removing the vertices~$\xx_1$ from~$F$, consists of isolated vertices and vertex-disjoint copies of the graph~${H - V(G)}$, which, by \refL{lem:StrBal}~\ref{eq:StrBal:connected}, is connected. 
Let~$H'$ be obtained from~${H - E(G)}$ by removing isolated root vertices (if any). 
Since~$(G,H)$ is grounded, we have ${e_{{H - V(G)}} < e_{H'}}$. 
Note that~$H'$ is connected (since it equals~${H - V(G)}$ with some root vertices connected to it) and therefore~${F - \xx_1}$ is~$H'$-free. 
It follows that any extension of~$\xx_2$ that is present in~$F$ must intersect~$\xx_1$, so there are at most~$|\xx_1| = v_G$ such vertex-disjoint extensions of~$\xx_2$. 
\end{proof}
\begin{claim}\label{cl:conditionalsum}%
Let~$\xx_1,\xx_2 \in \osets$ be disjoint. Then 
\begin{equation}\label{eq:sumP}
\sum_{\cC_1 \in \fH(\xx_1)}\sum_{\cC_2 \in \fH(\xx_2, \cC_1)}\Pr(\cI_{\cC_2} \mid \cI_{\cC_1}) 
\; \le \; 
(1+o(1)) \sum_{\cC_1 \in \fH(\xx_1)}\sum_{\cC_2 \in \fH(\xx_2)}\Pr(\cI_{\cC_2}) .
\end{equation}
\end{claim}
\begin{proof}[Proof of Claim~\ref{cl:conditionalsum}]%
In the first step we estimate~$\sum_{\cC_2 \in \fH(\xx_2, \cC_1)}\Pr(\cI_{\cC_2} \mid \cI_{\cC_1})$ using a counting argument 
that accounts for the different kinds of overlaps of~$\cC_2$ with the graph~$F = F(\cC_1)$ defined in~\eqref{eq:F}. 
Turning to the details, as in the proof of Claim \ref{cl:D2:upper} we will think of~$(G,H)$-extensions as~edge-sets of size~${e_H-e_G}$. 
Recall that $|\cC_1| = |\cC_2| = z = \ceil{(1 + \eps)\mu}$. 
Suppose that the graph~$F$ contains~$k$ extensions of~$\xx_2$.
If $\cC_2 \in \fH(\xx_2,\cC_1)$ then all these~$k$ extensions must be present in $\cC_2$, since otherwise $\Pr(\cI_{\cC_1 \cup \cC_2} , \cD_{\cC_1^c \cup \cC_2^c}) \le \Pr(\cI_{\cC_1}, \cD_{\cC_2^c})= 0$ contradicting $\cC_2 \in \fH(\xx_2, \cC_1)$. 
List the remaining extensions~in~$\cC_2$ as $E_1, \dots, E_{\dex - k}$ in an arbitrary order.
Note that each~$E_i$ is not fully contained in~$E(F)$, and thus the intersection~$E_i \cap E(F)$ is the edge-set of some $(G,J_i)$-extension of~$\xx_2$ for some graph~$J_i$ satisfying~$G \subseteq J_i \subsetneq H$ 
(the case~$J_i = G$ occurs when the extension~$E_i$ is edge-disjoint from~$F$). 
When these intersections are given by~${J_1, \dots, J_{\dex-k}}$, then we clearly have 
\begin{equation*}
\prob{ \cI_{\cC_2} \mid \cI_{\cC_1} } = \prod_{i = 1}^{\dex - k} p^{e_H - e_G - (e_{J_i}-e_G)} = \prod_{i = 1}^{\dex - k} p^{e_H - e_{J_i}}.
\end{equation*}
Furthermore, the number of sequences $E_1, \dots, E_{z - k}$ corresponding to intersections~${J_1, \dots, J_{\dex-k}}$ is bounded from above by
\begin{equation*}%\label{eq:sumcount} 
\prod_{i = 1}^{\dex - k} \bigpar{v_G + (\vGH)\dex}^{v_{J_i} - v_G}\extcount{J_i,H},
\end{equation*}
where~$\extcount{J,H} := N_{G,H}=N$ if~$G = J$ and~$\extcount{J,H} := n^{v_H - v_J}$ otherwise. 
%Noting~$z!/(z-k)! \le z^k$, 
Hence, summing over all possible choices of~$J_1, \dots, J_{\dex-k}$ and dividing by~$(\dex - k)!$ (since we sum over unordered collections~$\cC_2$), it follows that 
\begin{align}
	\sum_{\cC_2 \in \fH(\xx_2, \cC_1)}\Pr(\cI_{\cC_2} \mid \cI_{\cC_1}) & \le   \notag \frac{1}{(\dex-k)!} \sum_{\substack{J_1, \dots, J_{\dex-k}:\\ G \subseteq J_i \subsetneq H}} \prod_{i = 1}^{\dex - k} \bigpar{v_G + (\vGH)\dex}^{v_{J_i} - v_G} \extcount{J_i,H} p^{e_H - e_{J_i}} 
	\\
	&\le \frac{\dex^k}{\dex!} \cdot \biggpar{\sum_{G \subseteq J \subsetneq H}\bigpar{v_G + (\vGH)\dex}^{v_J - v_G} \extcount{J,H} p^{e_H - e_J}}^{\dex - k}
  \label{eq:rhs}.
\end{align}
Noting that $\extcount{G,H}p^{e_H - e_G} = \mu$, using \eqref{eq:Harris:overlap} and~$\mu \asymp \dex$ we bound the sum in~\eqref{eq:rhs} from above by, say, 
\begin{equation}\label{eq:muerror}
	\mu + O\bigpar{\sum_{G \subsetneq J \subsetneq H}  \dex^{v_H}n^{v_H - v_J}p^{e_H - e_J}} 
	\le \mu + o(1)  = \mu \cdot \Bigpar{1 + o\bigpar{\dex^{-1}}}.
\end{equation}
From the assumptions~$\eps \le 1$ and~$\mu \ge 1/2$ 
it follows that~$\dex \le (1+\eps)\mu + 1 \le 4 \mu$, say. 
Therefore, in view of \eqref{eq:rhs}--\eqref{eq:muerror}, using~$\mu=Np^\eGH$ and~\eqref{eq:Hr} it follows~that 
\begin{equation}\label{eq:intermediate}
	\sum_{\cC_2 \in \fH(\xx_2, \cC_1)}\Pr(\cI_{\cC_2} \mid \cI_{\cC_1}) \le \left( \frac{\dex}{\mu}\right)^k\frac{(Np^\eGH)^{\dex}}{\dex!} \Bigpar{1 + o\bigpar{\dex^{-1}}}^{\dex-k} 
\: \le \: 
(1+o(1)) \cdot 4^k|\fH(\xx_2)|p^{(\eGH)\dex} ,
\end{equation}
whenever the graph~$F$ defined in~\eqref{eq:F} contains exactly~$k$ extensions of~$\xx_2$.

In the second step we sum the above estimate~\eqref{eq:intermediate} over all~$\cC_1 \in \fH(\xx_1)$. 
Recalling the partition~\eqref{eq:part}, note that~$k =0$ when~$\cC_1 \in \fH_0(\xx_1,\xx_2)$, and that~$k \le v_G$ otherwise (see Claim~\ref{cl:finiteOverlaps}). 
From~\eqref{eq:intermediate} it follows that 
\[
\sum_{\cC_1 \in \fH(\xx_1)}\sum_{\cC_2 \in \fH(\xx_2, \cC_1)}\Pr(\cI_{\cC_2} \mid \cI_{\cC_1}) \le (1+o(1)) \cdot \Bigpar{|\fH_0(\xx_1,\xx_2)|  + 4^{v_G}|\fH_{\ge 1}(\xx_1,\xx_2)|} \cdot |\fH(\xx_2)|p^{(\eGH)\dex} .
\]
In view of~\eqref{eq:negligbleC1}, the factor in the above parentheses is at most~$(1+o(1)) \cdot |\fH(\xx_1)|$, say,  
which together with~$p^{(\eGH)\dex}=\Pr(\cI_{\cC_2})$ from~\eqref{eq:I} then completes the proof of inequality~\eqref{eq:sumP}.  
\end{proof}
Finally, inserting the estimates~\eqref{eq:sumP}, $p^{(\eGH)\dex}=\Pr(\cI_{\cC_1})$, and~\eqref{eq:D:lower} into~\eqref{eq:EE1}, 
it follows that 
\begin{equation*}%\label{eq:EE2}
\Pr(\cE_{\xx_1},\cE_{\xx_2}) \; \le \; (1+o(1)) \sum_{\cC_1 \in \fH(\xx_1)}\Pr(\cI_{\cC_1})\Pr(\cD_{\cC_1^c} | \cI_{\cC_1})\sum_{\cC_2 \in \fH(\xx_2)} \Pr(\cI_{\cC_2})\Pr(\cD_{\cC_2^c} | \cI_{\cC_2}) ,
\end{equation*}
which together with~\eqref{eq:er} completes the proof of inequality~\eqref{eq:pr:ub} and thus Lemma~\ref{lem:main} 
(which in turn implies the \zerost{} in~\eqref{eq:main:strbal} of \refT{thm_strictly_balanced}, as discussed). \noproof

\subsection{The $1$-statement}\label{sec:1statement} 
Our proof of the \onest{} in~\eqref{eq:main:strbal} of Theorem~\ref{thm_strictly_balanced} is based on a fairly standard union bound argument. 
\begin{proof}[Proof of the $1$-statement in~\eqref{eq:main:strbal} of Theorem~\ref{thm_strictly_balanced}]% 
Fix an arbitrary constant~$\tau > 0$. 
For~$\beta > 0$ as given by \refL{lem:StrBal}~\ref{eq:StrBal:density}, fix constants~$0 < \gamma \le \beta$ and $0 < \alpha < \gamma/2$ as in the proof of the $0$-statement (see \refS{sec:0statement}). 
If~$p > n^{-1/d(G,H) + \gamma}$, then \refR{rem:Phibig}~\ref{eq:Phibig:iv} implies~$\Phi_{G,H} = \Omega(n^\gamma)$, and using~$\eps^2 \Phi_{G,H} = \Omega(n^{\gamma - 2\alpha}) = n^{\Omega(1)}$ 
we see that the $1$-statement of Theorem~\ref{thm_strictly_balanced} follows from Theorem~\ref{thm_generaltail}~\ref{thm_tail1} with~$t = \eps \mu$.

In the remaining (main) case~$p \le n^{-1/d(G,H) + \gamma}$, we fix a root~$\xx \in \osets$. 
Since there are~$O(n^{v_G})$ many such roots, 
for the $1$-statement of Theorem~\ref{thm_strictly_balanced} it suffices to show that, for~$C>0$ large enough,
\begin{equation}\label{eq:thm_ext_01:goal}
	\prob{|X_\xx - \mu| \ge \eps \mu} = o\bigl( n^{- (v_G + \tau)} \bigr) \qquad \text{ if~$\eps^2 \mu \ge C \log n$.} 
\end{equation}
To avoid clutter, we shall henceforth use the convention that all implicit constants~$c_i$ may depend on~$(G,H)$. 
For the lower tail we shall apply Janson's inequality~\cite[Theorem~1]{RiordanWarnke2015} 
analogously to the textbook argument~\cite{JLR,JW} for unrooted subgraph counts, 
which in view of~\eqref{eq:lem:density:subs} from \refL{lem:StrBal}~\ref{eq:StrBal:density} 
routinely gives 
\begin{equation}\label{eq:1:LT:Janson}
	\prob{X_\xx \le (1- \eps) \mu} 
\; \le \;  \exp\Bigl(-c_1 \eps^2 \mu\Bigr) 
\le n^{-c_1 C} = o\bigl( n^{- (v_G + \tau)} \bigr)
\end{equation} 
for~$C > (v_G+\tau)/c_1$ 
(similar to~\eqref{eq:Janson:Delta} and~\eqref{eq:Janson:Delta:bound}, the relevant~$\Delta$-term of Janson's inequality, 
which here is defined in terms of the family~$\cS$ of edge-sets corresponding to extensions of~$\xx$ in~$K_n$, satisfies~$\Delta = o(1)$ by~\eqref{eq:mumupper} and~\eqref{eq:lem:density:subs}). 
For the upper tail we shall apply~\cite[Theorem~32]{WUT} in the setting described in~\cite[Example~20]{WUT} 
(the conditions (H$\ell$), (P), (P$q$) are defined in~\cite[Section~4.1]{WUT}). 
The underlying hypergraph $\cH = \fH(\xx)$ consists of the edge-sets of extensions of $\xx$, thus having vertex-set~$V(\cH) = E(K_n)$. 
We set the parameters to~$N = n^2$, $\ell = 1$, $q = k = \eGH$, and~$K=v_G+2\tau$. 
The quantity~$\mu_j$ from~\cite[Example~20]{WUT} satisfies~$\max_{1 \le j < q}\mu_j \le \max_{G \subsetneq J \subsetneq H} n^{v_H - v_J}p^{e_H - e_J} \ll n^{-\beta}$ by \refL{lem:StrBal}~\ref{eq:StrBal:density}. 
Invoking~\cite[Theorem~32]{WUT}, it then follows~that
\begin{equation}\label{eq:1:UT:Wsmallp}
\prob{X_\xx \ge (1+\eps) \mu} \; \le \;  \bigl(1 + o(1)\bigr) \cdot \exp\Bigl(-\min\bigl\{c_2\eps^2 \mu, \: (v_G + 2\tau)\log n\bigr\}\Bigr) = o\bigl( n^{- (r + \tau)} \bigr)
\end{equation}
for~$C > (v_G+\tau)/c_2$, completing the proof of~\eqref{eq:thm_ext_01:goal} 
 and thus the $1$-statement in~\eqref{eq:main:strbal} of Theorem~\ref{thm_strictly_balanced}.  
\end{proof}
\begin{remark}[Theorem~\ref{thm_strictly_balanced}: stronger $1$-statement]\label{rem:thm_strictly_balanced}% 
The above proof yields, in view of~\refR{rem:thm_generaltail}, the following stronger conclusion: 
for any fixed~$\tau > 0$ there is a constant~$C=C(\tau,G,H)>0$ such that 
the~$1$-statement in~\eqref{eq:main:strbal} of Theorem~\ref{thm_strictly_balanced} 
holds with probability~$1 - o(n^{-\tau})$. 
\end{remark}

\subsection{Deferred proof of \refL{lem:StrBal}}\label{s_strictly_balanced:prelim:deferred}
For completeness, we now give the routine proof of \refL{lem:StrBal} 
deferred from \refS{s_strictly_balanced:prelim}. 
\begin{proof}[Proof of \refL{lem:StrBal}]% 
\ref{eq:StrBal:density}: Set~$\Psi_{J,H} := n^{v_H - v_J}p^{e_H - e_J}$. 
In the case $v_J = v_H$, for any~$\beta > 0$ satisfying~$1/d(G,H) > 2\beta$ 
we have~$\Psi_{J,H} = p^{e_H-e_J} \ll n^{-(e_H - e_J)\beta} \le n^{-\beta}$. 
Thus we can henceforth assume~$v_J < v_H$.
Since~$G$ is an induced subgraph of~$H$ and thus of~$J$, we also have~$v_G < v_J$. 
Since~$(G,H)$ is strictly balanced we have~$d(G,J)< d(G,H)$, which implies 
\begin{equation}\label{eq:lem:density:subs:3}
d(J,H) = \frac{(e_H-e_G)-(e_J-e_G)}{(v_H-v_G)-(v_J-v_G)} = \frac{(v_H-v_G)d(G,H)-(v_J-v_G)d(G,J)}{(v_H-v_G)-(v_J-v_G)} > d(G,H) .
\end{equation}
Hence~$1/d(G,H) > 1/d(J,H) + 2\beta$ for~$\beta>0$ sufficiently small,  
so that~$p = O(n^{-1/d(G,H)+\beta}) \ll n^{-1/d(J,H) - \beta}$. 
Observe that~$e_H > e_J$, since otherwise~$e_H = e_J$ and~$v_H > v_J$ imply~$d(G,J) > d(G,H)$, 
contradicting that~$(G,H)$ is strictly~balanced.
Hence $\Psi_{J,H} = (n^{1/d(J,H)} p)^{e_H-e_J} \ll n^{-\beta}$, 
completing the proof of~\eqref{eq:lem:density:subs}. 

\ref{eq:StrBal:connected}: Assume the contrary. 
Then we can split $V(H) \setminus V(G)$ into two nonempty sets~$V_1$ and~$V_2$ 
such that there are no edges between $V_1$ and $V_2$. 
Writing~$H_i := H[V(G) \cup V_i]$, we readily obtain
\begin{equation*}
		d(G,H) = \frac{e_H-e_G}{v_H-v_G} = \frac{\sum_{i \in [2]}(e_{H_i}-e_G)}{\sum_{i \in [2]}(v_{H_i}-v_G)} = \frac{\sum_{i \in [2]}(v_{H_i}-v_G)d(G,H_i)}{\sum_{i \in [2]}(v_{H_i}-v_G)} \le \max_{i \in [2]}d(G,H_i) .
	\end{equation*}
Since~$(G, H)$ is strictly balanced we have~$d(G,H_i)< d(G,H)$, yielding the desired contradiction. 
\end{proof}

\section{No grounded primals case (\refT{thm_nogrounded})}\label{s_nogrounded}
In this section we prove Theorem~\ref{thm_nogrounded} by focusing on a maximal primal subgraph~$J_{\max}$ of~$(G, H)$; 
we remark that~$J_{\max}$ is in fact unique (the union of all primal subgraphs), but we do not need this. 
Our arguments hinge on the basic observation that, since~$J_{\max}$ is by assumption not grounded (i.e.,~there are no edges between~$V(G)$ and~$V(J_{\max}) \setminus V(G)$), 
extension counts~$X_{G,J_{\max}}(\xx)$ are essentially the same as the number of \emph{unrooted} copies of the graph~$K := {J_{\max} - V(G)}$, where the vertices of~$G$ are~deleted from~$J_{\max}$.

For the $1$-statement this heuristically means that if~$X_{G,J_{\max}}(\xx)$ is concentrated for \emph{some}~$\xx$, then $X_{G,J_{\max}}(\xx)$ is concentrated for \emph{all}~$\xx$ 
(the reason being that not too many copies of~$K$ can overlap with any root~$\xx'$, see Lemma~\ref{lem:scattered} below). 
Furthermore, using Theorem~\ref{thm_generaltail}~\ref{thm_tail1} it turns out that \whp{} each copy of~$J_{\max}$ extends to the `right' number of $H$-copies 
%(the crux is that we are way above the existence threshold~$n^{1/m(J_{\max}, H)}$ of the rooted graph~$(J_{\max}, H)$, see Lemma~\ref{lem:nogroundeddensity} below). 
(here the crux will be that~$\Phi_{J_{\max}, H} = n^{\Omega(1)}$ follows from \refR{rem:Phibig}~\ref{eq:Phibig:iv} and Lemma~\ref{lem:nogroundeddensity} below). 
Combining these two estimates then allows us to deduce that~\whp{}~$X_{G,H}(\xx)$ is concentrated for all~$\xx$;  see Section~\ref{sec:ext:1} for~the~details.

For the $0$-statement we shall proceed similarly, the main difference is that, for a \emph{fixed}~$\xx$, we start by arguing that~$X_{G,J_{\max}}(\xx)$ is not concentrated, i.e., \whp{} far away from its expected value. 
This allows us to deduce that~$\xx$ has~\whp{} the wrong number of $(G,H)$-extensions (since by Theorem~\ref{thm_generaltail}~\ref{thm_tail1} \whp{} each copy of~$J_{\max}$ again extends to the right number of copies of~$H$); see Section~\ref{sec:ext:0} for~the~details.

\subsection{Setup and technical preliminaries}
In the upcoming arguments it will, as in~\cite{S90b}, often be convenient to treat extensions as sequences of vertices. Given a rooted graph~$(G,H)$ with labeled vertices~$V(G) = {\left\{ 1, \dots, v_G \right\}}$ and~${V(H) \setminus V(G)} = {\{v_G + 1, \dots, v_H\}}$, 
\emph{an ordered $(G,H)$-extension of} $\xx = {(x_1, \dots, x_{v_G})} \in \osets$ is a sequence $\yy = {(y_{v_G + 1}, \dots, y_{v_H})}$ of distinct vertices 
from~$[n] \setminus \{x_1, \dots, x_{v_G}\}$ %outside of~$\xx$ 
such that the injection which maps each vertex~${j \in V(G)}$ onto $x_j$ and each vertex~${i \in {V(H)\setminus V(G)}}$ onto~$y_i$, also maps every edge~${f \in {E(H) \setminus E(G)}}$ onto an edge. 
%(in other words, it corresponds to an ordered copy of $H_G=(V(H),E(H) \setminus E(G))$ in which each vertex~$i \in V(G)$ is mapped onto~$x_i$ and each vertex~$j \in V(H)\setminus V(G)$ is mapped onto~$y_j$). 
%
Given a root~$\xx \in \osets$, let~$Y_{G,H}(\xx)$ denote the number of ordered~$(G,H)$-extensions of~$\xx$ in~$\Gnp$. 
Note that
\begin{equation}\label{def:nuGH}
\nu_{G,H} := \E Y_{G,H}(\xx) = (n-v_G) (n-v_G-1) \cdots (n-v_H+1) \cdot p^{\eGH}  % \asymp n^{\vGH}p^{\eGH}.
\end{equation}
does not depend on the particular choice of~$\xx$.   
Let~$\aut(G,H)$ denote the number of automorphisms of~$H$ that fix the set~$V(G)$. 
Since each extension corresponds to~$\aut(G,H)$ many ordered extensions, we~obtain
\begin{align}
\label{eq:YX}
Y_{G,H}(\xx) &= \aut(G,H) \cdot X_{G,H}(\xx) , \\
\label{eq:numu}
	\nu_{G,H} &= \aut(G,H) \cdot \mu_{G,H},
\end{align}
where $\mu_{G,H} = \E X_{G,H}(\xx)$ is defined as in~\eqref{def:muGH}.
One further useful elementary observation is that, for any induced~$G \subseteq J \subseteq H$, we have 
\begin{equation}\label{eq:muprod}
	\nu_{G,J} \cdot \nu_{J,H} = \nu_{G,H}.
\end{equation}
Our arguments will also exploit the following technical property of maximal primal~subgraphs. 
\begin{lemma}\label{lem:nogroundeddensity}%
If~$J_{\max}  \subsetneq H$ is a maximal primal of the rooted graph~$(G,H)$, \linebreak[3] 
then~$m(J_{\max},H) < m(G,H)$.
\end{lemma}
\begin{proof}
Fix~$J_{\max}  \subsetneq J \subseteq H$. 
Using maximality of~$J_{\max} \supsetneq G$, 
we infer~$d(G, J) < m(G,H)$ and~$d(G, J_{\max} ) = m(G,H)$. 
Proceeding analogously to inequality~\eqref{eq:lem:density:subs:3}, it routinely follows that 
\[
d(J_{\max},J)  
%= \frac{(\er{G}{J})-(\er{G}{J_{\max}})}{(\vr{G}{J})-(\vr{G}{J_{\max}})} 
= \frac{(\vr{G}{J})d(G,J)-(\vr{G}{J_{\max}})d(G,J_{\max})}{(\vr{G}{J})-(\vr{G}{J_{\max}})}
< m(G,H),
\]
which completes the proof by maximizing over all feasible~$J$.
\end{proof}

\subsection{The $0$-statement}\label{sec:ext:0}
As discussed, for the $0$-statement of Theorem~\ref{thm_nogrounded} the core idea is 
to show that~$X_{G,J_{\max}}(\xx)$ is not concentrated for some~$\xx \in \osets$, 
and that~$X_{J_{\max},H}(\y)$ is concentrated for all~$\y \in \oset{v_{J_{\max}}}$, 
see~\eqref{eq:badJ}--\eqref{eq:sparser}~below. 
\begin{proof}[Proof of the $0$-statement of Theorem~\ref{thm_nogrounded}]
	Assuming~$\eps \ge n^{-\alpha}$ with~$\alpha < 1/2$ (as we may), we have $\eps^2\Phi_{G,H} = \Omega(n^{1-2\alpha}p^{\eGH}) \gg p^{\eGH}$, so the assumption~$\eps^2\Phi_{G,H} \to 0$ implies~$p \to 0$ and thus~$1-p=\Theta(1)$.  
Since $(G,H)$ has no grounded primals, the desired $0$-statement now follows by combining
the conclusions of Theorem~\ref{thm_generaltail}~\ref{thm_tail0} for the cases~$\Phi_{G,H} \to 0$ and~$\Phi_{G,H} \to \infty$ 
with the conclusion of Lemma~\ref{lem:generalzero} below for~$\Phi_{G,H} \asymp 1$
(formally using, as usual, the subsubsequence principle~\cite[Section~1.2]{JLR}). %, completing the~proof. 
\end{proof}
\begin{lemma}\label{lem:generalzero}%
Let~$(G,H)$ be a rooted graph with no grounded primal subgraphs. 
Then, for all~$p=p(n) \in [0,1]$ and~$\eps=\eps(n) \in (0,1]$ with~$\Phi_{G,H} \asymp 1$ and~$\eps \to 0$,
\begin{equation}\label{eq:lem:generalzero}
\lim_{n \to \infty} \Pr\Bigpar{\max_{\xx \in \osets}|X_\xx - \mu| \ge \eps \mu} = 1 .
\end{equation}
\end{lemma}
\begin{proof}%
Note that by increasing~$\eps$ if necessary, we may henceforth assume $\eps \ge n^{-\alpha}$ for any constant~$\alpha > 0$ 
(since increasing~$\eps$ can only decrease the probability on the left-hand side of~\eqref{eq:lem:generalzero} above). 
Let~$J_{\max}$ be a maximal primal subgraph of~$(G,H)$. 
By \refR{rem:Phibig}~\ref{eq:Phibig:ii}--\ref{eq:Phibig:iii}, the assumption~$\Phi_{G,H} \asymp 1$~implies 
\begin{gather}	
\label{eq:muGJmax}
\mu_{G,J_{\max}} \asymp 1 ,\\
\label{eq:conditionPhip}
p = \Omega\bigl(n^{-1/m(G,H)}\bigr) .
\end{gather}

Turning to the details, we start with the claim that, \whp{}, 
\begin{align}
\label{eq:badJ}
\max_{\xx \in \osets} |X_{G,J_{\max}}(\xx) - \mu_{G,J_{\max}}| & > 3\eps \mu_{G,J_{\max}}, \\
\label{eq:sparser}
\max_{\y \in \oset{v_{J_{\max}}}} |X_{J_{\max},H}(\y)  - \mu_{J_{\max},H}|  & < \tfrac{1}{2}\eps \mu_{J_{\max},H} . 
\end{align}
To show that this claim implies the desired $0$-statement, 
we consider ordered extensions and note that multiplying~\eqref{eq:badJ} and~\eqref{eq:sparser} by $\aut(G, J_{\max})$ and $\aut(J_{\max},H)$, respectively, we can replace~$X$~by~$Y$ and~$\mu$~by~$\nu$, cf.~\eqref{eq:YX} and \eqref{eq:numu}. 
Observe that each ordered $(G,H)$-extension corresponds to a unique pair of extensions: one of~$\xx$ with respect to~$(G,J_{\max})$ and one of~$\y$ (which consists of~$\xx$ plus the vertices of the first extension) with respect to $(J_{\max},H)$. 
Consequently, recalling the identity~\eqref{eq:muprod}, inequalities~\eqref{eq:badJ}--\eqref{eq:sparser} imply that there is~$\xx \in \osets$ such that either
\begin{equation}\label{eq:YGHR:1}
	Y_{G,H}(\xx) > (1 + 3\eps)\nu_{G,J_{\max}} \cdot (1 - \eps/2) \nu_{J_{\max},H} > (1 + \eps) \nu_{G,H}
\end{equation}
or 
\begin{equation}\label{eq:YGHR:2}
	Y_{G,H}(\xx) < (1 - 3\eps)\nu_{G,J_{\max}} \cdot (1 + \eps/2) \nu_{J_{\max},H} < (1 - \eps) \nu_{G,H} ,
\end{equation}
which in view of~\eqref{eq:YX} and \eqref{eq:numu} establishes the desired $0$-statement (after rescaling by~$\aut(G,H)$).

It remains to show that~\eqref{eq:badJ} and \eqref{eq:sparser} hold \whp{}, and we start with~\eqref{eq:badJ}. 
Consider the unrooted graph~$K := {J_{\max} - V(G)}$, where the vertices of~$G$ are deleted from~$J_{\max}$. 
By construction, we have~$v_K = \vr{G}{ J_{\max}}$.  
Since~$J_{\max}$ is not grounded, we also have~$e_K = \er{G}{ J_{\max}}$. 
Using~\eqref{eq:muGJmax} we infer 
\begin{equation}\label{eq:muK:Jmax}
\mu_K \asymp n^{v_K}p^{e_K} = n^{\vr{G}{J_{\max}}}p^{\er{G}{J_{\max}}} \asymp \mu_{G,J_{\max}} \asymp 1 ,
\end{equation}
which by Markov's inequality implies that the number of $K$-copies is \whp{} at most~$n/(2v_K)$, say (with room to spare). 
This means that either (i)~there are no $K$-copies, in which case $X_{G,J_{\max}}(\xx) = 0$ for all $\xx \in \osets$, or (ii)~the $K$-copies span at most~$n/2$ vertices, in which case there is one $\xx_1 \in \osets$ that is disjoint from all $K$-copies and another set~$\xx_2 \in \osets$ that intersects at least one $K$-copy, so that~$X_{G,J_{\max}}(\xx_1) = X_K > X_{G,J_{\max}}(\xx_2)$. 
In both cases it follows that~\eqref{eq:badJ} holds \whp{}, since~\eqref{eq:muGJmax} and~$\eps \to 0$ imply that the interval~$(1 \pm 3\eps)\mu_{G,J_{\max}}$ does not contain zero, and moreover contains at most one~integer.

Turning to~\eqref{eq:sparser}, note that~\eqref{eq:sparser} holds trivially when~$J_{\max} = H$. 
Otherwise~$m(J_{\max},H) < m(G,H)$ by Lemma~\ref{lem:nogroundeddensity}, 
so that~\eqref{eq:conditionPhip} implies $p = \Omega( n^{\gamma-1/m(J_{\max},H)} )$ for some constant $\gamma > 0$. 
Using \refR{rem:Phibig}~\ref{eq:Phibig:iv}, it follows that~$\Phi_{J_{\max},H} = \Omega(n^\gamma)$. 
Assuming~$\eps \ge n^{-\alpha}$ with~$\alpha < \gamma / 2$ (as we may), we infer $\eps^2\Phi_{J_{\max},H} =  \Omega(n^{\gamma/ 2 -\alpha}) = n^{\Omega(1)}$. 
Applying Theorem~\ref{thm_generaltail}~\ref{thm_tail1} with~$t = \tfrac{1}{2}\eps \mu_{J_{\max},H}$, 
now~\eqref{eq:sparser} holds~\whp{}. 
\end{proof}
%
%Only the proof of~\eqref{eq:badJ} used that no primal subgraph is grounded, 
%so variants of the above argument can be adapted to other special cases, in particular when every primal subgraph is grounded 
%(then the event that all roots~$\xx$ must have the same (positive) number of extensions to~$J_{\max}$ implies that 
%the number~$X_L$ of unrooted copies of the graph~$L := {J_{\max} - E(G)}$ equals one specific value, 
%and such point probabilities are~$o(1)$ by asymptotic normality of~$X_L$). 

\subsection{The $1$-statement}\label{sec:ext:1}
As discussed, for the $1$-statement of Theorem~\ref{thm_nogrounded} we rely on the fact that  no vertex is contained in too many copies of the (unrooted) graph~${J_{\max} - V(G)}$, which is formalized by \refL{lem:scattered} below. 
As usual, given a graph~$K$ with~$v_K \ge 1$, 
subgraphs~$J \subseteq K$ with~$v_J \ge 1$ that maximize the density $d_J := d(\emptyset, J) = e_J/v_J$ are called~\emph{primal} (consistently with rooted graphs terminology), 
and~$K$ is called~\emph{balanced} when~$K$ itself is~primal.  
\begin{lemma}\label{lem:scattered}%
Let $K$ be a balanced graph with $e_K \ge 1$. 
There are constants $\beta, C > 0$ 
such that, for all $p=p(n) \in [0,1]$ with~$n^{\beta - 1/d_K} \ll p = O(n^{\beta - 1/d_K})$, 
in $\Gnp$ \whp{} every vertex~$x \in [n]$ is contained in at most~$C\lambda^{\vr{G_{\min}}{K}}$ copies of~$K$, where~$\lambda := np^{d_K}$ 
and~$G_{\min} \subseteq K$ is a primal subgraph with the smallest number of vertices. 
\end{lemma}
\noindent
We defer the density based proof to \refApp{apx:scattered} 
(which is rather tangential to the main~argument here), 
and now use \refL{lem:scattered} to prove the desired $1$-statement~of~Theorem~\ref{thm_nogrounded}. 
\begin{proof}[Proof of the $1$-statement of~\refT{thm_nogrounded}]%
The assumptions~$\eps \le 1$ and~$\eps^2\Phi_{G,H} \to \infty$ imply~$\Phi_{G,H} \to \infty$, so \refR{rem:Phibig}~\ref{eq:Phibig:i} implies~$p\gg~n^{-1/m(G,H)}$. 
If~$\eps^2\Phi_{G,H} = n^{\Omega(1)}$, then the desired $1$-statement follows from Theorem~\ref{thm_generaltail}~\ref{thm_tail1}, 
so we may further assume~$\eps^2 \Phi_{G,H} \le n^c$ for any constant~$c>0$ of our choice, 
which together with the assumption~$\eps \ge n^{-\alpha}$ implies $\Phi_{G,H} \le n^{c + 2\alpha}$. 
Using the contrapositive of \refR{rem:Phibig}~\ref{eq:Phibig:iv}, by choosing~$\alpha,c>0$ sufficiently small (as we may) 
we thus henceforth can~assume 
\begin{equation}	\label{eq:conditionPhi}
	n^{-1/m(G,H)} \ll p \ll n^{\beta - 1/m(G,H)},
\end{equation}
where the constant~$\beta > 0$ is as given by~Lemma~\ref{lem:scattered}.

Turning to the details, let~$J_{\max}$ be a maximal primal subgraph of~$(G,H)$. 
For convenience we use ordered extensions, as before.  
Note that~$Y_{G,J_{\max}}(\xx)$ is the number of (unrooted) copies of graph~$K := {J_{\max} - V(G)}$ that are disjoint from~$\xx$. 
For any vertex~$x \in [n]$, let~$Z_K(x)$ denote the number of copies of~$K$ containing~$x$. 
We fix some~$\xx' \in \osets$, and start with the claim that there exists a constant $D>0$ such that, \whp{}, 
\begin{align}
\label{eq:XRGprime}
	|Y_{G,J_{\max}}(\xx') - \nu_{G,J_{\max}}| &< \tfrac 1 8\eps\nu_{G,J_{\max}}. \\
\label{eq:goodcount}
\max_{\y \in \oset{v_{J_{\max}}}} |Y_{J_{\max},H}(\y)  - \nu_{J_{\max},H}| &< \tfrac{1}{2}\eps \nu_{J_{\max},H},\\
\label{eq:vertexbounded}
	\max_{x \in [n]} Z_K(x) & \le D\frac{\eps\nu_{G, J_{\max}}}{\eps^2\Phi_{G,H}}.
\end{align}
We now show that this claim implies the desired $1$-statement. 
In view of~\eqref{eq:XRGprime}, the first step is to use~\eqref{eq:vertexbounded} to show that $Y_{G,J_{\max}}(\xx)$ is also concentrated for the remaining roots $\xx \neq \xx'$. 
Specifically, using~\eqref{eq:vertexbounded} to bound the number of $(G, J_{\max})$-extensions of $\xx$ that overlap with $\xx'$ (and those of $\xx'$ overlapping with $\xx$),
in view of the assumption~$\eps^2\Phi_{G,H} \to \infty$ it follows that, for every~$\xx \in \osets$,  
\begin{equation*}
	|Y_{G,J_{\max}}(\xx) - Y_{G,J_{\max}}(\xx')| 
	\le O\Bigpar{\sum_{x \in \xx \cup \xx'} Z_K(x)} \ll \tfrac 1 8 \eps\nu_{G,J_{\max}} .
\end{equation*}
Together with~\eqref{eq:XRGprime} this implies that, say, 
\begin{equation}\label{eq:XRG}
	\max_{\xx \in \osets} |Y_{G,J_{\max}}(\xx) - \nu_{G,J_{\max}}| < \tfrac 1 4 \eps\nu_{G,J_{\max}} .
\end{equation}
The second step exploits that by~\eqref{eq:goodcount} each copy of $J_{\max}$ extends to the `right' number of copies of~$H$.
Indeed, with analogous reasoning as for~\eqref{eq:YGHR:1}--\eqref{eq:YGHR:2} from \refS{sec:ext:0}, by combining~\eqref{eq:goodcount} with~\eqref{eq:XRG} 
it now follows (in view of~\eqref{eq:muprod}) that 
\[
	\max_{\xx \in \osets} Y_{G,H}(\xx) < (1 + \eps/4)\nu_{G,J_{\max}} \cdot (1 + \eps/2)\nu_{J_{\max},H} < (1 + \eps)\nu_{G,H} ,
\]
and similarly, 
\[
   \min_{\xx \in \osets} Y_{G,H}(\xx) > (1 - \eps)\nu_{G,H} ,
\]
which in view of~\eqref{eq:YX}--\eqref{eq:numu} establishes the $1$-statement of Theorem~\ref{thm_nogrounded} (by rescaling by $\aut(G,H)$).

It remains to show that~\eqref{eq:XRGprime}--\eqref{eq:vertexbounded} hold \whp{}, and we start with~\eqref{eq:XRGprime}. 
Since $\Phi_{G,J_{\max}} \ge \Phi_{G,H}$ by definition, 
using Chebyshev's inequality together with the variance estimate~\eqref{eq:Variance} and $\eps^2 \Phi_{G,H} \to \infty$, it follows that
\begin{equation*}%\label{eq:XRGprime:old}
\begin{split}
\prob{|X_{G,J_{\max}}(\xx') - \mu_{G,J_{\max}}| \ge \tfrac 1 8\eps\mu_{G,J_{\max}}}  
	& \le \frac{\Var X_{G,J_{\max}}(\xx)}{(\eps/8)^2\mu_{G,J_{\max}}^2} 
	\asymp \frac{1 - p}{\eps^2\Phi_{G,J_{\max} }} \le \frac{1}{\eps^2\Phi_{G,H}} \to 0 ,
\end{split}
\end{equation*}
which in view of~\eqref{eq:YX}--\eqref{eq:numu} then implies that~\eqref{eq:XRGprime} holds \whp{} (by rescaling by $\aut(G,J_{\max})$).

Next we establish~\eqref{eq:goodcount}. Note that the proof of~\eqref{eq:sparser} only relies on~\eqref{eq:conditionPhip} (which here holds by \eqref{eq:conditionPhi}), and that we may assume~$\eps \ge n^{-\alpha}$ for sufficiently small~$\alpha>0$. 
Hence by the same argument as for~\eqref{eq:sparser}, 
in view of~\eqref{eq:YX}--\eqref{eq:numu} it follows that~\eqref{eq:goodcount} holds~\whp{} (after rescaling by~$\aut(J_{\max},H)$).

Finally, we turn to the auxiliary estimate~\eqref{eq:vertexbounded}. 
Note that every subgraph~$J \subseteq K$ with~$v_J \ge 1$ satisfies~$d_J = d(G,G \cup J)$. 
Hence~$J$ is primal for~$K$ if and only if~$G \cup J$ is primal for~$(G,J_{\max})$. 
Since~$J_{\max}= G \cup K$ is primal for~$(G,H)$, it follows that~$K$ is balanced, with~$d_K=d(G,J_{\max})=m(G,H)$. 
Using assumption~\eqref{eq:conditionPhi}, we thus have~$n^{-1/d_K} \ll p \ll n^{\beta-1/d_K}$. 
Invoking Lemma~\ref{lem:scattered}, there is a constant~$C>0$ such that,~\whp{},     
\[
	\max_{x \in [n]} Z_K(x) \le C \lambda^{\vr{G_{\min}}{K}},
\]
where $G_{\min}$ is a primal subgraph of~$K$ with the smallest number of vertices. In particular, we have~$d_K=d_{G_{\min}}$. 
Since~$J_{\max}$ is a vertex-disjoint union of the graphs~$K$ and~$G$, % (since~$J_{\max}$ is primal and not grounded) 
using~${G_{\min} \subseteq K}$ we infer that~$e_{G_{\min}} = {\er{G}{G \cup G_{\min}}}$ and~$v_{G_{\min}} ={\vr{G}{G \cup G_{\min}}}$. 
Recalling~$\lambda = np^{d_K} = np^{d_{G_{\min}}}$, it now follows analogously to~\eqref{eq:muK:Jmax}~that 
\begin{equation*}
	\lambda^{\vr{G_{\min}}{K}} = \frac{n^{v_K} p^{e_K}}{n^{v_{G_{\min}}} p^{e_{G_{\min}}}} \asymp \frac{\mu_K}{\mu_{G_{\min}}} \asymp  \frac{\mu_{G, J_{\max}}}{\mu_{G,G \cup G_{\min}}} \le \frac{\mu_{G, J_{\max}}}{\Phi_{G,H}} ,
\end{equation*}
which together with~$1 \le 1/\eps = \eps/\eps^2$ and~\eqref{eq:numu} completes the proof of~\eqref{eq:vertexbounded} for suitable~$D>0$. 
\end{proof}

\section{Further cases}\label{sec:further}

\subsection{Unique and grounded primal case (\refT{thm_unique})}\label{s_unique}
In this section we prove \refT{thm_unique} by adapting the arguments from \refS{s_nogrounded} (focusing on the unique primal~$J=J_{\max}$). 
The key difference is that here we can use the $0$- and $1$-statements of our main result \refT{thm_strictly_balanced} to deduce that~$X_{G,J}(\xx)$ is not concentrated for some~$\xx$, or concentrated for all~$\xx$, respectively.  
This then allows us to prove the desired $0$- and $1$-statements, since each copy of~$J$ again extends to the `right' number of copies of~$H$ (by Theorem~\ref{thm_generaltail}~\ref{thm_tail1}, as in Section~\ref{s_nogrounded}); see \eqref{eq:JHconc}--\eqref{eq:GJconc} and~\eqref{eq:GJoff} below. 
\begin{proof}[Proof of \refT{thm_unique}]%
If~$\Phi_{G,H} \to 0$, then the $0$-statement holds by \refT{thm_generaltail}~\ref{thm_tail0}. 
Therefore we henceforth can assume~$\Phi_{G,H} = \Omega(1)$, which by \refR{rem:Phibig}~\ref{eq:Phibig:ii} is equivalent to
\begin{equation}\label{eq:overthreshold}
	p = \Omega\bigl(n^{-1/m(G,H)}\bigr).
\end{equation}
Note that the proof of~\eqref{eq:sparser} relies only on~\eqref{eq:conditionPhip} (which is the same as \eqref{eq:overthreshold}), the fact that $J_{\max}$ is the maximal primal (which also holds trivially for~$J$ in the current setting), 
and that we may assume $\eps \ge n^{-\alpha}$ for sufficiently small~$\alpha>0$ (which we may also assume here). 
Hence by the same argument as for~\eqref{eq:sparser}, 
after rescaling by~$\aut(J,H)$ (see~\eqref{eq:YX}--\eqref{eq:numu}) we here obtain that, \whp{}, 
\begin{equation}\label{eq:JHconc}
	\max_{\y \in \oset{v_J}}|Y_{J,H}(\y) - \nu_{J,H}| < \tfrac{1}{2} \eps \nu_{J,H}. 
\end{equation}

We start with the~$1$-statement. Since~$\mu_{G,J} \ge \Phi_{G,H}$ by definition, the assumption~$\eps^2\Phi_{G,H} \ge C \log n$ implies~$\eps^2\mu_{G,J} \ge C \log n$. By uniqueness of the primal~$J$, the rooted graph $(G,J)$ is strictly balanced. 
Therefore~\eqref{eq:main:strbal} of~\refT{thm_strictly_balanced} implies (after rescaling by~$\aut(G,J)$) that, for suitable~$\alpha, C > 0$,~\whp{}, 
	\begin{equation}\label{eq:GJconc}
		\max_{\xx \in \osets} |Y_{G,J}(\xx) - \nu_{G,J}| < \tfrac{1}{4}\eps \nu_{G,J}. 
	\end{equation}
The $1$-statement of \refT{thm_unique} now follows from~\eqref{eq:JHconc} and~\eqref{eq:GJconc} 
by exactly the same reasoning with which~\eqref{eq:goodcount} and~\eqref{eq:XRG} from Section~\ref{sec:ext:1} 
implied the $1$-statement of \refT{thm_nogrounded}.

We now turn to the~$0$-statement. We again plan to apply~\eqref{eq:main:strbal} of~\refT{thm_strictly_balanced} to the strictly balanced rooted graph~$(G,J)$, 
for which we need to check that the assumption~$\eps^2 \Phi_{G,H} \le c \log n$ implies the required condition~$\eps^2 \mu_{G,J} \le c\log n$. 
We will do this by showing that~$\Phi_{G,H} = \mu_{G,J}$ for~$n$ large enough. 
First, note that the assumptions~$\eps \ge n^{-\alpha}$ and~$\eps^2 \Phi_{G,H} \le c \log n$ imply~$\Phi_{G,H} = O(n^{2\alpha}\log n)$. 
By~\eqref{eq:overthreshold} and the contrapositive of \refR{rem:Phibig}~\ref{eq:Phibig:iv} we can thus assume that, say,  
\begin{equation}\label{eq:closetohreshold}
	p \asymp n^{\theta - 1/m(G,H)} \quad \text{ with } \quad \theta = \theta(n,p) \in [0,3\alpha].
\end{equation}
Since the primal~$J$ is unique, we have~$d(G,J)=m(G,H)$, and~$d(G,K)< m(G,H)$ when~$G \subsetneq K \subseteq H$ satisfies $J \neq K$.  
Hence there exists a constant~$\gamma=\gamma(G,J,H)>0$ % that is independent of~$\alpha$, 
such that, for any $G \subsetneq K \subseteq H$, 
	\begin{equation}\label{eq:mu:minimal}
		\mu_{G,K} \asymp \left( np^{d(G,K)} \right)^{\vr{G}{K}} \asymp \left( n^{1 - \frac{d(G,K)}{m(G,H)} + \theta d(G,K)} \right)^{\vr{G}{K}} = \begin{cases}
			\Omega(n^{\gamma}) & \text{if~$K \neq J$,} \\
			O(n^{3\alpha (\er{G}{J})}) & \text{if~$K = J$.} 
		\end{cases}
	\end{equation}
By taking~$\alpha > 0$ small enough, it follows that~$\Phi_{G,H} = \mu_{G,J}$ for~$n$ large enough, 
which (as discussed) establishes~$\eps^2 \mu_{G,J} \le c \log n$. 
Therefore~\eqref{eq:main:strbal} of~\refT{thm_strictly_balanced} implies (after rescaling by $\aut(G,J)$) that, \whp{},  
\begin{equation}\label{eq:GJoff}
	\max_{\xx \in \osets} |Y_{G,J}(\xx) - \nu_{G,J}| \ge 3\eps \nu_{G,J}. 
\end{equation}
The $0$-statement of \refT{thm_unique} now follows from~\eqref{eq:GJoff} and~\eqref{eq:JHconc} 
by the same (routine) reasoning with which~\eqref{eq:badJ}--\eqref{eq:sparser} from Section~\ref{sec:ext:0} 
implied the $0$-statement of \refT{thm_nogrounded}. 
\end{proof}
\begin{remark}[Theorem~\ref{thm_unique}: stronger $1$-statement]\label{rem:thm_unique1} %  
The above proof yields, in view of Remarks~\ref{rem:thm_generaltail}--\ref{rem:thm_strictly_balanced}, the following stronger conclusion: 
for any fixed~$\tau > 0$ there is a constant~$C=C(\tau,G,H)>0$ such that 
the~$1$-statement in~\eqref{eq:thm:unique} of Theorem~\ref{thm_unique} 
holds with probability~$1 - o(n^{-\tau})$. 
\end{remark}

\subsection{Strictly balanced and ungrounded case (\refT{thm_strictly_balanced})}\label{sec:ext:non}
In this section we prove the threshold~\eqref{eq:main:strbal:non} of~\refT{thm_strictly_balanced}~(ii) for strictly balanced rooted graphs~$(G,H)$ that are not grounded, % (this degenerate case is less interesting), 
which turns out to be a simple corollary of~\refT{thm_nogrounded}. 
The crux is that, by decreasing~$\alpha>0$ (if~necessary), we can ensure that the {$0$-}~and {$1$-statement} conditions in~\eqref{eq:main:strbal:non} and~\eqref{eq:main:nogrounded} coincide. 
\begin{proof}[Proof of~\eqref{eq:main:strbal:non} of~\refT{thm_strictly_balanced}]
Recall that~$\mu=\mu_{G,H}$ and $\Phi = \Phi_{G,H}$ are defined in~\eqref{def:muGH} and~\eqref{eq_PhiGH}, respectively. 
By assumption the unique primal~$H$ is not grounded, so \refT{thm_nogrounded} applies.
Decreasing the constant~$\alpha>0$ from \refT{thm_nogrounded}, we can assume that~$\beta \ge 3\alpha$, where~$\beta>0$ is the constant given by \refL{lem:StrBal}~\ref{eq:StrBal:density}. 
We now distinguish two ranges of~$p=p(n)$. 
First, when~$p \le n^{-1/d(G,H)+\beta}$, then~\eqref{eq:lem:density:subs} from \refL{lem:StrBal} implies that~$\mu=\Phi$ for~$n$ large enough (since~\eqref{eq:lem:density:subs} implies~$\mu_{G,H}/\mu_{G,J} \asymp n^{v_H - v_J}p^{e_H - e_J} \ll 1$ for all~$G \subseteq J \subsetneq H$ with~$e_J > e_G$). 
Second, when~$p \ge n^{-1/d(G,H)+\beta} \ge n^{-1/m(G,H)+3\alpha}$, then~$\eps \ge n^{-\alpha}$ and \refR{rem:Phibig}~\ref{eq:Phibig:iv} imply that~$\min\{\eps^2 \mu, \eps^2\Phi\} \ge n^{-2\alpha} \cdot \Phi = \Omega(n^{\alpha}) \to \infty$. 
Since in both ranges the {$0$-}~and {$1$-statement} conditions in~\eqref{eq:main:strbal:non} and~\eqref{eq:main:nogrounded} coincide, 
it follows that \refT{thm_nogrounded} implies~\eqref{eq:main:strbal:non}. 
\end{proof}

\section{Cautionary examples (\refP{prop:counterexample:2} and~\ref{prop:counterexample:1})}\label{sec:counterexample} 
In this section we prove Propositions~\ref{prop:counterexample:2}--\ref{prop:counterexample:1} 
for the rooted graphs~(e) and~(f) depicted in~\refF{counterexample}.  
The proof idea for~\refP{prop:counterexample:2} is to proceed in two rounds for a fixed vertex~$x$: 
using~\refT{thm_strictly_balanced} we first find about~$\mu_{G,K_4}$ many~$(G,K_4)$-extensions of~$\xx=(x)$, 
which we then extend to about~$\mu_{G,H}$ many~$(G,H)$-extensions of~$\xx$.  
The crux is that most of the relevant~$(K_4,H)$-extensions from the second round evolve nearly independently, 
which ultimately allows us to surpass the conditions of Spencer's result~\eqref{eq_Spencer_SB} and~\refT{thm_strictly_balanced} for~$(K_4,H)$. 
\begin{proof}[Proof of \refP{prop:counterexample:2}]% 
Recalling~$\omega = np^2$, by assumption we have~$\eps^2 \mu_{G,K_4} \asymp \eps^3\omega^3 \gg \log n$ and~$\eps^2 \mu_{K_4,H} \le \mu_{K_4,H} \asymp \omega \ll \log n$, which readily implies~$\log \omega \asymp \log \log n$ and~$p=n^{-1/2+o(1)}$. 
Now it is not difficult to verify that~$\Phi_{G,H} \asymp \mu_{G,K_4} \asymp \omega^3$ (either directly, or similarly as for~\eqref{eq:mu:minimal} from \refS{s_unique}). 
Turning to the details of the $1$-statement, 
we start with the auxiliary claim that, whp, for each vertex~$x$ the following event~$\cP_x$~holds:%
\begin{romenumerate2}
\item\label{eq:Pv:i}% [(i)]
The vertex-neighbourhood~$\Gamma_x$ of~$x$ has size~$|\Gamma_x| \le 9np$.
\item\label{eq:Pv:ii}% [(ii)] 
The collection~$\cT_x$ of all triangles spanned by~$\Gamma_x$ %(i.e., with all three endpoints inside $\Gamma_v$) 
has size~$|\cT_x| = (1\pm \eps/9) \tbinom{n-1}{3}p^6$.
\item\label{eq:Pv:iii}%[(iii)] 
Every vertex~$y \in \Gamma_x$ is contained in at most~$D:=15$ triangles from~$\cT_x$. 
\end{romenumerate2}
Indeed, invoking the $1$-statement of \refT{thm_strictly_balanced} with~$H$ equal to~$K_4$ and~$G$ being the root vertex~$v$, % (which is easily checked to be strictly balanced and grounded), 
from $\eps ^2 \mu_{G,K_4} \asymp \eps^2\omega^3 \gg \log n$ it follows that, whp,~\ref{eq:Pv:ii} holds for all vertices~$x$.
Since~$np = n^{1/2+o(1)} \gg \log n$, using standard Chernoff bounds it is routine to see that, whp,~\ref{eq:Pv:i} holds for all vertices~$x$. 
We claim that if~\ref{eq:Pv:iii}~fails for some~$y \in \Gamma_x$, then there are $4$~triangles in~$\cT_x$ containing $y$ that form either a flower (share no vertices other than~$y$) or a book (all contain an edge~$yz$ for some~$z \in \Gamma_x$): 
to see this, note that if we assume the contrary, then for a maximal flower (with at most~$3$ triangles) each edge of it is contained in at most~$2$ other $\cT_x$-triangles, whence there are at most~$3 + 6 \cdot 2 = 15$ triangles in~$\cT_x$ containing~$y$.
The probability that there is either a \mbox{$4$-flower} or \mbox{$4$-book} with all vertices connected to some extra vertex~$x$ is at most $n^{10}p^{21} + n^7p^{16} = n^{-1/2 + o(1)} \to 0$.
It follows that, whp, properties~\ref{eq:Pv:i}--\ref{eq:Pv:iii} hold for all vertices~$x$, establishing the~claim. 

We now fix a root vertex~$x$, and expose the edges of~$\Gnp$ in two rounds: 
in the first round we expose all edges incident to~$x$ and all edges inside~$\Gamma_x$, 
and then in the second round we expose all remaining edges. 
We henceforth condition on the outcome of the first exposure round, and assume that~$\cP_x$ holds. 
As usual, to avoid clutter we shall omit this conditioning from our notation. 
Given distinct vertices~$a,b \in \Gamma_x$, 
let~$Y_{a,b}$ denote the number of common neighbours of~$a$ and~$b$ in~$[n] \setminus (\{x\} \cup \Gamma_x)$.  
Note that~$\eps \omega \gg \sqrt{\log n/\omega} \gg 1$ by assumption. 
Since, by~\ref{eq:Pv:ii}, $|\cT_x| \asymp n^3p^6 = \omega ^3 \ll \eps \omega^4 \asymp \eps \mu$,
using~\ref{eq:Pv:iii} it is not difficult to see that
\begin{equation}\label{eq:Xx}
Z_x := \sum_{abc \in \cT_x}(Y_{a,b}+Y_{b,c} + Y_{a,c}) \qquad \text{satisfies} \qquad \bigl|X_{(x)}-Z_x\bigr| \: \le \:  3|\cT_x| \cdot D \ll \eps \mu/2 .
\end{equation}
Using~\ref{eq:Pv:i} and~$\eps \omega \gg 1$ (see above) we infer $1+|\Gamma_x| \le 10np = n^{1/2+o(1)} \ll n/\omega \ll \eps n$, 
and together with~\ref{eq:Pv:ii} it then follows that, say,  
\begin{equation}\label{eq:EZv} % n-4
\E Z_x = 3|\cT_x| \cdot \bigpar{n-1-|\Gamma_x|} p^2 = (1\pm \eps/8) \mu . 
\end{equation}
In~\eqref{eq:Xx} we now write each~$Y_{a',b'}$ as a sum of indicators of length~$2$~paths,  
which enables us to estimate the lower tail of~$Z_x$ via Janson's inequality. 
By distinguishing between pairs of edge-overlapping paths that share one or two endpoints, 
using~\ref{eq:Pv:iii} it is standard to see that the relevant~$\Delta$ term is at most~$\E Z_x \cdot (2D p + 2D) = O(\E Z_x)$, say.  
Using~$\eps^2 \E Z_x \asymp \eps^2 \omega^4 \gg \log n$, 
by invoking~\cite[Theorem~1]{RiordanWarnke2015} it then follows~that 
\begin{equation}\label{eq:Zv:upper}
\Pr(Z_x \le (1-\eps/8) \E Z_x) \le \exp\bigpar{- \Omega(\eps^2 \E Z_x)} = o(n^{-1}) .
\end{equation}
Using~\ref{eq:Pv:iii} we also see that any path shares an edge with a total of at most~$2 D = O(1)$ paths, 
which enables us to estimate the upper tail of~$Z_v$ via concentration inequalities for random variables with `controlled dependencies'. 
In particular, by invoking~\cite[Proposition~2.44]{JLR} (see also~\cite[Theorem~9]{W14}) it follows~that 
\begin{equation}\label{eq:Zv:lower}
\Pr(Z_x \ge (1+\eps/8) \E Z_x) \le \exp\bigpar{- \Omega(\eps^2 \E Z_x)} = o(n^{-1}) .
\end{equation}
To sum up, \eqref{eq:Xx}--\eqref{eq:Zv:lower} and~$1-\eps/2 < (1 \pm \eps/8)^2 < 1 + \eps/2$ imply~$\Pr(|X_{(x)}-\mu| \ge \eps \mu \mid \cP_x) = o(n^{-1})$, 
which readily completes the proof of the desired $1$-statement (since, whp,~$\cP_x$ holds for all~$n$ vertices~$x$). 
\end{proof}

The proof idea for \refP{prop:counterexample:1} is to find a copy of~$K_4$ with an edge that is contained in extremely many triangles. 
To this end we proceed in two steps, inspired by~\cite[Lemma~3]{SW18}: 
in the first step we find~$\Theta(n)$ many vertex-disjoint copies of~$K_4$, 
and in the second step we then find the desired edge contained in many~triangles. 
\begin{proof}[Proof of \refP{prop:counterexample:1}]%
Note that~$\mu \asymp \omega^5$.  
As in the proof of \refP{prop:counterexample:2}, we again have~$\log \omega \asymp \log \log n$ and $\Phi_{G,H} \asymp \mu_{G,K_4} \asymp \omega^3$,  
so~$\eps^2 \Phi_{G,H}\asymp \eps^2 \mu_{G, K_4} \gg \log n$ by assumption.  
Noting~$0.39< 2/5$, we~define
\begin{equation}\label{eq:CE:defz}
z := \Bigceil{2 \bigpar{(1+\eps) \mu}^{1/2}} \asymp \omega^{5/2} = o(\log n/\log \omega) .
\end{equation}

Turning to the details of the desired $0$-statement, %In the first step, 
let~$Y_{K_4}$ denote the size of the largest collection of vertex-disjoint copies of~$K_4$ spanned by the vertices in~$W:=\{1, \ldots, \lfloor n/2 \rfloor \}$.
It is routine to check that the minimum of~${|W|^{v_G}p^{e_G}}$, taken over all~${G \subseteq K_4}$ with~${v_G \ge 1}$, equals~$|W| \approx n/2$ for~$n$ large enough. % (resulting from~$G$ being an isolated vertex). 
Since~the induced subgraph of $\Gnp$ spanned by~$W$ has the same distribution as~$\G_{|W|,p}$, 
by invoking~\cite[Theorem~3.29]{JLR} there is a constant~$c>0$ such~that 
\begin{equation}\label{eq:XK4v}
  \Pr(Y_{K_4} \ge cn) = 1 - o(1).
\end{equation}

%In the second step, we
We now condition on the edges spanned by~$W$, and assume that~$Y_{K_4} \ge cn$. 
To avoid clutter, we shall again omit this conditioning from our notation (as in the proof of \refP{prop:counterexample:1}).  %(to avoid clutter).
We henceforth fix~$\ceil{cn}$ vertex-disjoint copies of~$K_4$ spanned by~$W$, 
and from the~$i$-th such copy %of~$K_4$ 
we pick an edge~${\{v_i,w_i\}}$ and a further vertex~${x_i \not\in \{v_i,w_i\}}$. 
Defining~$Z_i$ as the number of vertices in~$[n] \setminus W$ that are common neighbours of~$v_i$ and~$w_i$, 
using~$\tbinom{m}{z} \ge (m/z)^z$ for~$m \ge z$ together with~$np^2 = \omega = o(z)$ and~\eqref{eq:CE:defz}, it routinely follows~that 
\begin{equation*}
\Pr(Z_i \ge z) 
\ge 
% \Pr(Z_i = z) =
\binom{\lceil n/2 \rceil}{z}p^{2z} \bigpar{1-p^2}^{\lceil n/2 \rceil-z} \ge \Bigpar{\frac{np^2}{2z}}^{z} \e^{-np^2} = \e^{-\Theta(z\log \omega)} \ge n^{-o(1)} . 
\end{equation*}
Note that~$Z_i \ge z$ implies $X_{(x_i)} \ge \binom{z}{2} \ge \bigpar{z/2}^2 \ge (1+\eps) \mu$. 
Since the random variables~$Z_i$ depend on disjoint sets of independent edges, it then follows that 
\begin{equation*}%\label{eq:eq:XK4v2}
	\Pr(\max_{x \in [n]}X_{(x)} < (1+\eps)\mu) \le \Pr(\max_{1 \le i \le \lceil cn \rceil}\hspace{-0.25em}Z_i < z) = \hspace{-0.125em}\prod_{1 \le i \le \lceil cn \rceil} \hspace{-0.375em} \Pr(Z_i < z) \le \Bigpar{1-n^{-o(1)}}^{\lceil cn \rceil} = o(1) .
\end{equation*}
Hence~$\Pr(\max_{x \in [n]}X_{(x)} \ge (1+\eps)\mu \mid Y_{K_4} \ge cn ) = 1-o(1)$, 
which together with~\eqref{eq:XK4v} completes the~proof. 
\end{proof}
%%
%\noindent
%%
%In the setting of \refP{prop:counterexample:1} it is easy 
%(by combining Janson's inequality with~$\eps^2\Phi \gg \log n$) 
%to deduce $\Pr(\min_{x \in [n]}X_{(x)} \ge (1-\eps)\mu)=1-o(1)$, 
%so only the upper tail of~$X_\xx$ causes difficulties. 

\pagebreak[3]

\section{Concluding remarks}\label{sec:conclusion}
The results and problems of this paper can also be viewed through the lens of \emph{extreme value theory}, 
where a standard goal is to show that a (suitably shifted and normalized) maximum converges to a non-degenerate distribution. 
To see the connection, note that the proof of \refT{thm_strictly_balanced}~(i) describes an interval on which~$\max_{\xx \in \osets} X_\xx$ is \whp{} concentrated. 
%  the random variable~$\max_{\xx \in \osets} X_\xx$
Our setting concerns discrete random variables (which can have complicated behaviour, cf.~\cite[Section 8.5]{FHR}), 
with a correlation structure that seems quite unusual for the field. 
Hence, as a first step, it would already be interesting to establish a `law of large numbers' result 
(even for a restricted class of~$(G,H)$, such as strictly balanced~ones), 
which is the content of the following~problem. 
\begin{problem}\label{prb:extreme_val}
Determine for what rooted graphs~$(G,H)$ and edge probabilities~$p = p(n)$ there is a sequence~$(a_n)$ of real positive numbers such that~${(\max_{\xx} X_{\xx} - \mu)/a_n}$ converges to~$1$ in probability (as~$n \to \infty$). 
\end{problem}
%
%We conjecture that for strictly balanced grounded $(G,H)$ the answer to the Problem~\ref{prb:extreme_val} holds with $a_n = \left(C\cdot\Var X_\xx \cdot \log n\right)^{1/2}$ for some constant $C = C(G,H)$, at least when $X_\xx$ is asymptotically normal.

\bigskip{\noindent\bf Acknowledgements.}
We are grateful to the referees for helpful suggestions concerning the presentation.

%\pagebreak[3]

\small
\bibliographystyle{plain}

\normalsize

%\pagebreak[3]

\begin{appendix}

\section{Appendix: Proof of \refT{thm_generaltail}}\label{apx:general} 
Our proof of \refT{thm_generaltail} from \refS{s_prelim} hinges on the following fairly routine claim (based on central moment estimates), 
where we write~$X_\xx = X_{G,H}(\xx)$, as usual. 
Recall that~$\mu=\E X_\xx$ and~$\sigma^2 = \Var X_\xx$ do not depend on the particular choice of~$\xx$, 
and that~$\Phi= \Phi_{G,H}$ is defined in~\eqref{eq_PhiGH}. 
\begin{claim}\label{cl:mom}%
For any rooted graph~$(G,H)$, the following holds for all~$p=p(n) \in [0,1]$ and~$\xx \in \osets$:%
\begin{romenumerate}
\item\label{cl:mom:order}%
If~$\Phi = \Omega(1)$, then~$\E (X_\xx - \mu)^{m} =	O\bigpar{ (\mu^2/\Phi)^{m/2}}$ for any fixed integer~$m \ge 2$, where the implicit constant may only depend on~$m$,~$G$ and~$H$.
\item\label{cl:mom:asymp}%
If~$\Phi(1-p) \to \infty$, then~$\prob{|X_\xx - \mu| < \delta \sigma} \to \prob{|\eta| < \delta}$ for any fixed~$\delta \in(0,\infty)$, 
where~$\eta$ is a standard normal random variable. 
\end{romenumerate}%
\end{claim}
\begin{proof}[Proof of Claim~\ref{cl:mom}]%
Recalling the variance estimate~\eqref{eq:Variance} 
and the definition~\eqref{eq_PhiGH} of~$\Phi$, a straightforward extension of the textbook proof of~\cite[Theorem~6.5 and Remark~6.6]{JLR} for (unrooted) subgraph counts~yields 
\begin{equation}\label{eq:centralmoment}
	\E (X_\xx - \mu)^{m} = \indic{\text{$m$ even}}(m-1)!!\sigma^{m}(1+o(1))  + O\Bigpar{\sum_{1 \le \ell < m/2} \sigma^{m}\bigpar{\Phi(1-p)}^{\ell-m/2}},
\end{equation}
where~$(m-1)!! = (m-1)\cdot(m-3) \cdots 1$ when~$m$ is even.
As~$\ell-m/2 < 0$ the sum in~\eqref{eq:centralmoment} is~$o(\sigma^m)$ when~$\Phi(1-p) \to \infty$. 
Hence~$\E(X_\xx-\mu)^m/\sigma^m \to \indic{m\text{ even}}(m-1)!!$, 
so the method of moments (see, e.g.,~\cite[Corollary~6.3]{JLR}) 
implies that~$(X_\xx - \mu)/\sigma$ converges to~$\eta$ in distribution, 
which implies~\ref{cl:mom:asymp}.

Turning to~\ref{cl:mom:order}, from~\eqref{eq:Variance} and~$\Phi = \Omega(1)$ 
we infer that in~\eqref{eq:centralmoment} we have~$\sigma^m = O\bigpar{\mu^m /\Phi^{m/2}}$ and 
\[
\sigma^{m} \bigl(\Phi(1-p)\bigr)^{\ell-m/2} \asymp \mu^{m}/\Phi^{m/2} \cdot \Phi^{\ell-m/2} (1-p)^{\ell}  = O\bigpar{\mu^{m}/\Phi^{m/2}}, 
\]
which shows that \eqref{eq:centralmoment} implies~\ref{cl:mom:order}.
\end{proof}
\begin{proof}[Proof of \refT{thm_generaltail}]% 
For~\ref{thm_tail1} we may assume that in our lower bound on~$t$ we have~$(t/\mu)^2 \Phi \ge n^{c}$ for sufficiently large~$n$, where~$c>0$ is a constant.   
Fix an arbitrary constant~$\tau > 0$, and set~$m := \lceil (v_G+\tau+1)/c\rceil$.
Using first a union bound, next Markov's inequality, and finally~\refCl{cl:mom}~\ref{cl:mom:order},  
it then readily follows~that 
\begin{equation*}
	\prob{\max_{\xx \in \osets} |X_\xx - \mu| \ge t} \le \sum_{\xx \in \osets}\frac{\E (X_\xx - \mu)^{2m}}{ t^{2m} } 
	\le O(n^{v_G}) \cdot \biggl(\frac{\mu^2}{t^2\Phi}\biggr)^m 
	= o(n^{-\tau }).
\end{equation*}

Turning to~\ref{thm_tail0} we fix~$\xx = (1, \dots, v_G)$, say, and claim that~$|X_\xx - \mu| \ge \eps \mu$~\whp{}.
In~case~(a) we fix~$\delta > 0$. Combining the variance estimate~\eqref{eq:Variance} with our assumption~$\eps^2 \Phi/(1-p) \to 0$, we infer for~$n$ large enough that 
\[
\eps \mu/\sigma  = \sqrt{\eps^2 \mu^2/\sigma^2} \asymp \sqrt{\eps^2 \Phi/(1-p)} \ll \delta .
\]
Together with \refCl{cl:mom}~\ref{cl:mom:asymp}, it follows that 
\[
	\limsup_{n \to \infty }\prob{|X_\xx - \mu | < \eps \mu} \le \prob{|\eta| < \delta}.
\]
Since~$\delta > 0$ was arbitrary, now the basic fact~$\lim_{\delta \to 0}\prob{|\eta| \le \delta} \to 0$ completes the proof in case~(a).
In the remaining~case~(b), after recalling~$X_\xx = X_{G,H}(\xx)$, then Markov's inequality readily implies 
\begin{equation*}
	\Pr(X_\xx \ge 1) \le \min_{G \subsetneq J \subseteq H: e_J > e_G}\Pr(X_{G,J}(\xx) \ge 1) \le \Phi_{G,H} \to 0,
\end{equation*}
so that~\whp{}~$|X_\xx - \mu| = \mu \ge \eps \mu$, completing the proof. 
\end{proof}

\section{Appendix: Proof of \refL{lem:scattered}}\label{apx:scattered}
Recalling the notation and definitions for unrooted graphs introduced at the beginning of Section~\ref{sec:ext:1}, 
\refL{lem:scattered} is implied by claim~\ref{eq:sc:v} of the following more general auxiliary result, 
whose technical statement is optimized for ease of the~proofs (which are partially inspired by~\cite[Lemma~4~and~7]{S90a}). 
\begin{lemma}\label{lem:scattered:technical}
Let~$K$ be a balanced graph with~$e_K \ge 1$. 
There are constants~$\beta, B, C > 0$ such that,  
for all~$p=p(n) \in [0,1]$ with~$n^{-1/d_K} \ll p = O(n^{\beta - 1/d_K})$, 
the following holds \whp{} in~$\Gnp$, writing~$\lambda := np^{d_K}$:
\begin{romenumerate}
		\item\label{eq:sc:i}% [(i)]
		If $G \subseteq K$ is primal for~$K$, then any two copies of $G$ are either vertex-disjoint or their intersection is isomorphic to a primal subgraph of $K$.
		\item\label{eq:sc:ii}% [(ii)] 
		If $G_0 \subsetneq G_1$ are both primal for~$K$ and there is no third primal $F$ such that $G_0 \subsetneq F \subsetneq G_1$, then,   
		for every copy~$G_0'$ of~$G_0$, all copies of~$G_1$ that contain~$G_0'$ are vertex-disjoint outside of $V(G_0')$. 
		\item\label{eq:sc:iii}%[(iii)] 
If~$G_0, G_1$ are as in \ref{eq:sc:ii}, then every copy of~$G_0$ is contained in at most~$B \lambda^{v_{G_1} - v_{G_0}}$ copies of~$G_1$. 
		\item\label{eq:sc:iv}% [(iv)] 
			If~$v \in V(K)$ and $G^{(v)} \subseteq K$ is a minimal primal subgraph of~$K$ containing~$v$, then for each vertex~$x \in [n]$ there is at most one $(v,G^{(v)})$-extension of $x$.
		\item\label{eq:sc:v}% [(v)] 
		If~$G_{\min} \subseteq K$ is primal for~$K$ with the smallest number of vertices, then every vertex~$x \in [n]$ is contained in at most~$C\lambda^{\vr{G_{\min}}{K}}$ copies of~$K$. 
\end{romenumerate}
\end{lemma}
\begin{proof}%[Proof of \refL{lem:scattered:technical}]
\ref{eq:sc:i}: 
Fix a graph~$U := G_1 \cup G_2$ that is formed by the union of some two distinct overlapping copies~$G_1,G_2$ of~$G$. 
Since there are only finitely many such graphs, it is enough to show that $\Gnp$ \whp{} does not contain a copy of~$U$ when the intersection~$I := G_1 \cap G_2$ is not isomorphic to a primal subgraph of~$K$. 
Noting~$e_U = 2e_G - e_I$ and~$v_U = 2v_G - v_I$, using that~$I$ is not primal, i.e., $d_I < d_K = d_G$, it follows~that 
\[
  d_U = \frac{e_U}{v_U} = \frac{2e_G - e_I}{2v_G - v_I} = \frac{2v_Gd_G - v_Id_I}{2v_G - v_I} > d_G . 
\]
Since~$p =O(n^{\beta - 1/d_K}) \ll n^{-1/d_U}$ for~$\beta>0$ small enough, 
using $\mu_U \asymp n^{v_U}p^{e_U} = (n p^{d_U})^{v_U} \ll 1$ 
and Markov's inequality it readily follows that~$\Gnp$ \whp{} does not contain a copy of~$U$.

\ref{eq:sc:ii}: 
By~\ref{eq:sc:i} any two distinct copies of~$G_1$ that contain the same copy of~$G_0$ must intersect in a subgraph isomorphic to some primal~$J$ with $G_0 \subseteq J \subsetneq G_1$. 
The assumed properties of $G_0,G_1$ imply that~$J$ cannot contain~$G_0$ properly. 
Hence~$J = G_0$, which implies that all copies of~$G_1$ are vertex-disjoint outside~$V(G_0')$.

\ref{eq:sc:iii}: 
In $K_n$, for each a copy~$G_0'$ of~$G_0$ the number of copies of $G_1$ containing $G_0'$ is at most~${\big\lfloor An^{\vr{G_0}{G_1}}\bigr\rfloor}$ for some constant~$A = A(K) \ge 1$. 
Defining~$B := \e^2A$, let~$\cE_{G_0'}$ denote the event that there are at least ${z:=\big\lceil B \lambda^{v_{G_1} - v_{G_0}}\big\rceil}$ copies of~$G_1$ that (a)~contain~$G_0'$ and (b)~are vertex-disjoint outside of~$V(G_0')$. 
% (to motivate this technical definition, note that by~(ii) we can later assume that whp all such extensions are vertex-disjoint). 
Since the copies share no edges other than~$E(G_0')$, a standard union bound argument yields
\begin{equation*}
\begin{split}
	\prob{\cE_{G_0'}} \; \le \; \binom{\big\lfloor An^{\vr{G_0}{G_1}}\bigr\rfloor}{z} p^{e_{G_0} + (\er{G_0}{G_1})z} . 
\end{split}
\end{equation*}
Since $K$ is balanced and $G_0$, $G_1$ are its primal subgraphs, we see that~$d_{G_0} = d_{G_1} = d_K$,  
from which it routinely follows that 
\[
	d(G_0,G_1) = \frac{e_{G_1}-e_{G_0}}{v_{G_1}-v_{G_0}} = \frac{v_{G_1}d_{K} - v_{G_0}d_{K}}{v_{G_1} - v_{G_0}} = d_{K}.
\]
Hence~$n^{\vr{G_0}{G_1}}p^{\er{G_0}{G_1}} = (np^{d(G_0,G_1)})^{\vr{G_0}{G_1}} = \lambda^{\vr{G_0}{G_1}}$. 
Together with $\binom{x}{z} \le (\e x/z)^z$ and the definition of~$z$, it then follows that, say, 
\begin{equation*}
\begin{split}
	\prob{\cE_{G_0'}} & \le p^{e_{G_0}} \cdot \bigl(\e A \lambda^{\vr{G_0}{G_1}}/z\bigr)^z \le p^{e_{G_0}} \cdot \exp\bigl(-\lambda^{\vr{G_0}{G_1}}\bigr) .
\end{split}
\end{equation*}
Using~$d_{G_0} = d_K$, we also obtain~$n^{v_{G_0}}p^{e_{G_0}} = (np^{d_{G_0}})^{v_{G_0}} = \lambda^{v_{G_0}}$. 
Noting that~$\lambda \to \infty$ as~$n \to \infty$, 
by summing over all possible copies~$G_0'$ (there are at most~$n^{v_{G_0}}$ of them) it then follows that 
\[
\sum_{G_0'} \prob{\cE_{G_0'}} \le 
n^{v_{G_0}} p^{e_{G_0}}
\cdot \exp \bigl(-\lambda^{\vr{G_0}{G_1}}\bigr) \le \lambda^{v_{G_0}} \cdot \exp \bigl(- \lambda \bigr) \to 0 .
\]
Now~\ref{eq:sc:iii} follows readily by a union bound argument and~\ref{eq:sc:ii}.

\ref{eq:sc:iv}: 
Given~$x \in [n]$, let~$G', G''$ be two $(v,G^{(v)})$-extensions of~$x$. Since~$G', G''$ are both copies of a primal subgraph~$G^{(v)}$, by~\ref{eq:sc:i} we have that~$G' \cap G''$ is a copy of a primal subgraph~$J \subseteq G^{(v)}$. Since each of the two isomorphisms mapping~$G^{(v)}$ to~$G'$ and $G''$ maps~$v$ to~$x$, we infer that~$v \in V(J)$. But since~$G^{(v)}$ is minimal among primals containing~$v$, it follows that~$J = G^{(v)}$ and thus~$G' = G''$.

\ref{eq:sc:v}:
Given $v \in V(K)$, set $G_0 := G^{(v)}$ and choose a maximal chain 
\[
G^{(v)}= G_0 \subsetneq G_1 \subsetneq \dots \subsetneq G_{\ell} = K 
\]
of primal subgraphs of~$K$, with the property that for any~$i$ there is no primal~$F$ for~$K$ satisfying $G_i \subsetneq F \subsetneq G_{i + 1}$ 
(to clarify: since~$K$ is balanced and thus primal, such maximal chains always exist). 
For each $(v,K)$-extension of~$x$ we can select a unique sequence of copies 
\[
x \in G_0' \subsetneq G_1' \subsetneq \dots \subsetneq G_{\ell}' 
\]
such that~$G_0'$ is an $(v,G_0)$-extension of~$x$, and that, for each $i \in [\ell]$, $G_i'$ is a copy of $G_i$. 
Hence it is enough to bound the number of such sequences, assuming~\ref{eq:sc:iii} and~\ref{eq:sc:iv}. 
By~\ref{eq:sc:iv} there is at most one choice for~$G_0'$, 
and, given~$G_{i-1}'$ with~$i \in [\ell]$, by~\ref{eq:sc:iii} there are at most~$B\lambda^{\vr{G_{i-1}}{ G_i}}$ choices for suitable~$G_{i}'$.  
Multiplying these bounds, we obtain
 \begin{equation*}
	 \max_{x \in [n]} X_{v,K}(x) \le \prod_{i \in [\ell]} B \lambda^{\vr{G_{i-1}}{G_i}} \asymp \lambda^{\sum_{i \in [\ell]}[\vr{G_{i-1}}{G_i}]} = \lambda^{\vr{G_0}{G_\ell}} = \lambda^{\vr{G^{(v)}}{K}}.
 \end{equation*}
Summing over all~$v \in V(K)$, using~$\vr{G^{(v)}}{K} \le \vr{G_{\min}}{K}$ and~$\lambda \gg 1$ it follows that, for some $C = C(K)>0$, each vertex is contained in at most~$C\lambda^{\vr{G_{\min}}{K}}$ copies of~$K$, completing the proof of~\refL{lem:scattered:technical}.  
\end{proof}

\end{appendix}

\end{document}